\documentclass[reqno]{amsart}
\usepackage[scale=0.75, centering, headheight=14pt]{geometry}
\usepackage[latin1]{inputenc}
\usepackage[T1]{fontenc}
\usepackage{lmodern}
\usepackage[english]{babel}

\usepackage{amsmath,amssymb,amsfonts,amsthm}
\usepackage{mathtools,accents}
\usepackage{mathrsfs}
\usepackage{xfrac}
\usepackage{array} 
\usepackage{aliascnt}
\usepackage{booktabs} 
\usepackage{array} 

\usepackage{verbatim} 
\usepackage{subfig} 

\usepackage{mathrsfs, dsfont}
\usepackage{amssymb}
\usepackage{amsthm}
\usepackage{amsmath,amsfonts,amssymb,esint}
\usepackage{graphics,color}
\usepackage{enumerate}
\usepackage{mathtools,centernot}

\usepackage{microtype}
\usepackage{paralist} 
\usepackage{cases}
\usepackage[initials]{amsrefs}
\allowdisplaybreaks

\usepackage{braket}
\usepackage{bm}

\usepackage[citecolor=blue,colorlinks]{hyperref}
\addto\extrasenglish{}

\usepackage{enumerate}
\usepackage{xcolor}
\usepackage{aliascnt}

\makeatletter
\def\newaliasedtheorem#1[#2]#3{
  \newaliascnt{#1@alt}{#2}
  \newtheorem{#1}[#1@alt]{#3}
  \expandafter\newcommand\csname #1@altname\endcsname{#3}
}
\makeatother

\usepackage{indentfirst}

\usepackage{esint}
\usepackage{mathrsfs}
\usepackage{stmaryrd}

\DeclareUnicodeCharacter{00A0}{ }
\DeclareUnicodeCharacter{00A0}{~}

\makeatletter
\newsavebox{\measure@tikzpicture}

\DeclareMathOperator{\id}{id}

\newcommand{\R}{\mathbb{R}}
\newcommand{\eps}{\varepsilon}
\newcommand{\N}{\mathbb{N}}

\newcommand{\weak}{\overset{*}{\rightharpoonup}}
\newcommand{\A}{\mathcal{A}}
\newcommand{\G}{\mathbb{G}(N,m)}

\newcommand{\BBB}{B_{\Psi^*}(S^\perp)}

\DeclareMathOperator{\op}{op}
\DeclareMathOperator{\Ker}{Ker}
\DeclareMathOperator{\dist}{dist}
\DeclareMathOperator{\spn}{span}

\DeclareMathOperator{\rank}{rank}

\DeclareMathOperator{\dv}{div}
\DeclareMathOperator{\tr}{tr}
\DeclareMathOperator{\Tan}{Tan}
\DeclareMathOperator{\spt}{spt}

\DeclareMathOperator{\loc}{loc}

\DeclareMathOperator{\sign}{sign}

\DeclareMathOperator{\Lip}{Lip}

\DeclareMathOperator{\E}{\mathds{E}}

\theoremstyle{plain}
\newtheorem{introtheorem}{Theorem}

\newtheorem{Teo}{Theorem}[section]
\newtheorem{lemma}[Teo]{Lemma}
\newtheorem{prop}[Teo]{Proposition}

\newtheorem{Cor}[Teo]{Corollary}
\newtheorem{Def}[Teo]{Definition}
\newtheorem{remark}[Teo]{Remark}

\fontsize {11pt} {19pt}
\selectfont

\numberwithin{equation}{section}
\title{Regularity for graphs with bounded anisotropic mean curvature}

\author[A. De Rosa and  R. Tione]{Antonio De Rosa\and Riccardo Tione}

\address{Antonio De Rosa
\hfill\break Department of Mathematics, University of Maryland, 4176 Campus Dr, College Park, Maryland 20742, United States}
\email{anderosa@umd.edu}

\address{Riccardo Tione  
\hfill\break  EPFL B, Station 8, CH-1015 Lausanne, CH}
\email{riccardo.tione@epfl.ch}


\begin{document}
\maketitle
\begin{abstract}
We prove that $m$-dimensional Lipschitz graphs with anisotropic mean curvature bounded in $L^p$, $p>m$, are regular almost everywhere in every dimension and codimension. This provides partial or full answers to multiple open questions arising in the literature. The anisotropic energy is required to satisfy a novel ellipticity condition, which holds for instance in a $C^2$ neighborhood of the area functional. This condition is proved to imply the atomic condition. In particular we provide the first non-trivial class of examples of anisotropic energies in high codimension satisfying the atomic condition, addressing an open question in the field.
As a byproduct, we deduce the rectifiability of varifolds (resp. of the mass of varifolds) with locally bounded anisotropic first variation for a $C^2$ (resp. $C^1$) neighborhood of the area functional.
In addition to these examples, we also provide a class of anisotropic energies in high codimension, far from the area functional, for which the rectifiability of the mass of varifolds with locally bounded anisotropic first variation holds.
To conclude, we show that the atomic condition excludes non-trivial Young measures in the case of anisotropic stationary graphs.
\end{abstract}

\section{Introduction}
A celebrated theorem of W. Allard \cite{ALL} states that, given a rectifiable $m$-varifold $V$ in $\R^N$ with density greater or equal than $1$ and generalized mean curvature bounded in $L^p(\|V\|)$ with $p>m$, then $V$ is regular around $x\in \R^N$ provided $x$ has density ratio sufficiently close to $1$. The proof deeply relies on the monotonicity formula of the density ratio, which is strictly related to the special symmetries of the area functional, \cite{ALLMON}.
For this reason, it is an extremely hard and widely open question whether this result holds for anisotropic energies, \cite[Question 1]{DLDPKT}, i.e. assuming an $L^p$ bound on the anisotropic mean curvature, see \eqref{boundedmean} for the definition, with respect to functionals of the form
\begin{equation*}
\Sigma_\Psi(V):= \int_{\Gamma}\Psi(T_x\Gamma)\theta(z)d\mathcal{H}^m(z),\qquad \mbox{where  $V = (\Gamma,\theta)$ is a rectifiable $m$-varifold}.
\end{equation*}
To the best of our knowledge, the only available result is the regularity for codimension one varifolds with bounded generalized anisotropic mean curvature \cite{ARCATA} (further referred to as $\Psi$-mean curvature), under a \emph{density lower bound assumption} \cite[The basic regularity Lemma, Assumption (1)]{ARCATA}.

The aim of this paper is to provide an affirmative answer to the question above in any dimension and codimension in the case the varifold $V$ is associated to a Lipschitz graph, solving the open question \cite[Question 5]{DLDPKT} for graphs. Namely, we will prove the following main result, see Theorem \ref{regteo} (we refer the reader to Section \ref{s:notation} and Section \ref{s:condition} for notation):

\begin{introtheorem}\label{A}
Let $\Psi\in C^2(\G,(0,\infty))$ be a functional satisfying (USAC), let $p>m$ and consider an open, bounded set $\Omega\subset \R^m$. Let $u \in \Lip(\Omega,\R^n)$ be a map whose graph $\Gamma_u$ induces a varifold with $\Psi$-mean curvature $H$ in $L^p$ in $\Omega \times \R^n$. Then there exists $\alpha>0$ and an open set $\Omega_0$ of full measure in $\Omega$ such that
\[
u \in C^{1,\alpha}(\Omega_0,\R^n).
\]
\end{introtheorem}

Although there have been important results on the regularity of {\emph{minimizers}} for $\Sigma_\Psi$, \cite{Alm68,SSA,DPM,FIGAL,DUS}, the regularity for {\emph{stationary points}} of $\Sigma_\Psi$ is a completely open question in general codimension.
As mentioned above, our proof cannot rely on the monotonicity formula. Hence, we are not able to extend it to general rectifiable varifolds.
Instead, we introduce a novel ellipticity condition (USAC), which allows us to obtain a Caccioppoli inequality, giving an answer to \cite[Question 6]{DLDPKT}.

For the sake of exposition and without loss of generality, in this paper we will treat autonomous integrands as in Theorem \ref{A}. Nevertheless, we remark that Theorem \ref{A} can be easily extended to non autonomous integrands $\Psi\in C^2(\R^N\times\G,(0,\infty))$ satisfying (USAC) at every $x\in \R^N$, see Remark \ref{Remark:nonaut}. It is enough to observe that the first variation with respect to such integrands carries an additional term, which can be treated as part of the mean curvature term, see \cite[Equation (2.5)]{DDG}. 
\\

(USAC) reveals to be useful to tackle another open problem in the literature: providing non-trivial examples of anisotropic energies in general codimension satisfying the atomic condition (AC). Indeed, anisotropic energy functionals have attracted an increasing interest since the pioneering works of F. J. Almgren \cite{Alm68,Alm76}. In particular, the classical Almgren ellipticity (AE), (\cite[IV.1(7)]{Alm76} or~\cite[1.6(2)]{Alm68}), allowed Almgren to prove regularity for minimizers of anisotropic energies,~\cite{Alm68}. Very recently, an ongoing interest on the anisotropic Plateau problem has lead to a series of results,
see~\cite{HP2016gm,DDG2016,DDG2017,DDG2017b, DeR2016,FangKol2017,MOON}.
In~particular, in~\cite{DDG} G. De Philippis, the first author and F. Ghiraldin have introduced (AC) and proved it to be necessary and sufficient for the validity of the rectifiability of varifolds whose anisotropic first variation is a~Radon
measure. In codimension~one and in dimension~one,
(AC)~is proved to be equivalent to the strict convexity of the integrand, \cite{DDG}. However, in general codimension, characterizing (AC) in terms of more classical conditions (such as (AE), policonvexity, or others) remains an open problem, \cite[Page 2]{DDG}. The first author and S. Kolasi{\'{n}}ski have recently obtained one implication, proving that (AC) implies (AE), \cite{DRK}.
However, to date, in general codimension there are no examples of anisotropic energies (besides the area functional) satisfying (AC).  
We address this question in a result that we can roughly summarize as follows:
\begin{introtheorem}\label{B}
Integrands $\Psi$ in a $C^2$ neighborhood of the area functional satisfy (USAC); (USAC) implies (AC).
\end{introtheorem}
Hence, the anisotropic energies in a $C^2$-neighborhood of the area are the first functionals in the literature in general codimension to justify the regularity theory developed in \cite{DDG}. 
In particular, we deduce the rectifiability of varifolds with locally bounded anisotropic first variation for a $C^2$ neighborhood of the area functional.

(AC) can be relaxed to a condition (further denoted (AC1)), which  
is equivalent to the rectifiability of the mass of varifolds whose anisotropic first variation is a Radon measure, \cite{RDHR}. In codimension one, the convexity of the integrand implies (AC1), compare \cite[Section 3.3]{RDHR}. However, in general codimension, there are no non-trivial examples of anisotropic energies satisfying (AC1). 
We address this problem by proving:
\begin{introtheorem}\label{C}
Integrands $\Psi$ in a $C^1$ neighborhood of the area functional satisfy (AC1).
\end{introtheorem}
Theorem \ref{C} implies that, in codimension one, (AC1) is a strictly weaker notion than convexity of the integrand, see Remark \ref{remark:AC1}. This shows that the result of \cite[Page 656, point (b)]{RDHR} is indeed optimal. We also find a class of examples of functionals satisfying (AC1), which are not $C^1$-close to the area, see Theorem \ref{Lp}.
\\
\\
There are profound connections between anisotropic geometric variational problems and questions arising in the study of polyconvex energies, the latter being roughly speaking a parametric version of the former, see \cite[Page 229]{EVA}.
This link was investigated in \cite{DLDPKT,HRT,TR}. In particular, there is a \emph{canonical} way to associate a function $F_\Psi: \R^{n\times m}\to \R$ to an integrand $\Psi$ defined on $\mathbb{G}(N,m)$, in such a way that a Lipschitz graph $\llbracket\Gamma_u\rrbracket$ is stationary for $\Sigma_\Psi$ if and only if $u$ is stationary for the energy
\[
\E_{F_\Psi}(u) = \int_{\Omega} F_\Psi(Du(x))dx.
\]
A graph is said to be stationary if and only if it is critical for outer and inner variations. In \cite{TR}, the second author proves the regularity for 2-dimensional Lipschitz graphs that are stationary with respect to polyconvex integrands close to the area. This result is close in spirit to Theorem \ref{A}, but it carries deep differences: in \cite{TR}, the \emph{closeness} of the two functionals depends on the Lipschitz constant of the stationary graph, while in Theorem \ref{A} the closeness is more quantitative, depending just on $n$ and $m$. On the other hand, the proof of \cite{TR} yields full regularity for stationary points, and is based on completely different methods coming from the theory of differential inclusions, introduced in \cite{SAP}. 

In \cite{DLDPKT,HRT}, the second author, together with C. De Lellis, De Philippis, B. Kirchheim and J. Hirsch, investigated the possibility of constructing a nowhere regular stationary graph for $\Sigma_\Psi$, exploiting the convex integration techniques introduced by S. M\"uller \& V. \v Sver\'ak and L. Sz{\'{e}}kelyhidi in \cite{SMVS,LSP}. However it is proved that it is impossible to complete this task \emph{using the same strategy} of \cite{SMVS,LSP}, see \cite{DLDPKT,HRT}. In particular, the authors prove that if the polyconvexity of $F_\Psi$ complies with the stationarity of $u$, then one can exclude a certain type of Young measures, referred to as \emph{$T'_N$ configurations}, that proved to be the crucial tool in \cite{SMVS,LSP}. Here we show a much more systematic result in this direction:

\begin{introtheorem}\label{D}
(AC) excludes non-trivial Young measures in the case of stationary graphs.
\end{introtheorem}

Theorem \ref{D} provides an answer to \cite[Question 4]{DLDPKT}. In \cite[Question 1]{MSK}, B. Kirchheim, S. M\"uller and V. \v Sver\'ak leave as an open question to find rank-one convex functions whose differential inclusion associated to critical points (for outer variations only) supports only trivial Young measures. This question is largely open, and we provide here an answer in a neighborhood of the area (in arbitrary dimension and codimension), adding the hypothesis of criticality for inner variations.
To conclude, we remark that our regularity Theorem \ref{A} provides partial answers to questions that naturally arose in the context of quasiconvex energies, \cite[Page 65]{KRI}, and \cite[Question 2]{MSK}.
\\
\\
The paper is organized as follows. In Section \ref{s:notation} we recall some technical preliminaries concerning the theory of varifolds and Young measures. In Section \ref{s:condition} we introduce (SAC1) and (USAC), and we show their connection to (AC1) and (AC), thus proving Theorems \ref{B} and \ref{C}. Section \ref{s:regularity} is devoted to show the regularity of graphs with $L^p$-bounded anisotropic mean curvature with respect to a functional satisfying (USAC), i.e. Theorem \ref{A}. In Section \ref{s:compactness} we show the absence of nontrivial Young measures if (AC) holds, i.e. Theorem \ref{D}. Finally, in Section \ref{s:ex} we give an explicit example of a class of anisotropic energies in high codimension, far from the area functional, satisfying (AC1).

\section{Technical preliminaries and notation}\label{s:notation}

In this section, we recall the main definitions and results concerning varifolds and Young measures that we will need in the rest of the paper.

\subsection{Basic notation}
Given $A,B\in \R^{d\times d}$ and  $v,w\in \R^d$, we denote the inner products as $\langle A,B\rangle=\sum_{ij=1}^d A_{ij}B_{ij}$ and $(v,w)=\sum_{i=1}^d v_{i}w_{i}$. $\|A\|$ and $\|v\|$ will be the norms induced by the previous inner products. $A^t$ will denote the transpose of $A$.

\subsection{Measures  and rectifiable sets}\label{measures}
Given a locally compact separable metric space $Y$, we denote by \(\mathcal M(Y)\) the set of positive Radon measure on \(Y\). Given a Radon measure \(\mu\) we denote by \(\spt (\mu)\) its support. For a  Borel set \(E\),  \(\mu\llcorner E\) is the  restriction of \(\mu\) to \(E\), i.e. the measure defined by \([\mu\llcorner E](A)=\mu(E\cap A)\). 
Eventually, we denote by  $\mathcal{H}^m$ the $m$-dimensional Hausdorff measure.
\\
\\
A set \(K\subset \R^N\) is said to be \(m\)-rectifiable if it can be covered, 
up to an \(\mathcal{H}^m\)-negligible set, by countably many \(C^1\) $m$-dimensional 
submanifolds.  In the following we will only consider \(\mathcal{H}^m\)-measurable sets. Given an $m$-rectifiable set $K$, we denote with $T_xK$ the approximate tangent space of $K$ at $x$, which exists for $\mathcal{H}^m$-almost every point $x \in K$, \cite[Chapter 3]{SimonLN}.  A positive  Radon measure $\mu\in \mathcal M(\R^N)$ is said to be $m$-rectifiable if there exists an $m$-rectifiable set $K\subset \R^N$ such that $\mu= \theta\mathcal{H}^m \llcorner K$ for some Borel function \(\theta: \R^N\to (0,\infty)\).
\\
\\
For  \(\mu\in \mathcal M(\R^N) \) we consider its lower  and upper  \(m\)-dimensional densities at \(x\):
\[
\theta_*^m(x,\mu)=\liminf _{r\to 0} \frac{\mu(B_r(x))}{ \omega_m r^m}, \qquad \theta^{m*}(x,\mu)=\limsup_{r\to 0} \frac{\mu(B_r(x))}{ \omega_m r^m},
\]
where \(\omega_m\) is the volume of the \(m\)-dimensional unit ball in \(\R^m\). In case these  two limits are equal, we denote by \(\theta^m(x,\mu)\) their common value. Note that if $\mu= \theta\mathcal{H}^m \llcorner K$  with \(K\) rectifiable, then \(\theta(x)=\theta_*^m(x,\mu)=\theta^{m*}(x,\mu)\) for \(\mu\)-a.e. \(x\), see~\cite[Chapter 3]{SimonLN}.

If \(\eta :\R^n\to \R^n\) is a Borel map and \(\mu\) is a Radon measure, we let \(\eta_\# \mu=\mu\circ \eta^{-1}\) be the push-forward of \(\mu\) through \(\eta\). 

\subsection{Varifolds}\label{introvar}

We will use $\mathbb{G}(n + m,m)$ to denote the Grassmanian of (un-oriented) $m$-dimensional 
linear subspaces in $\R^{n + m}$ (often referred to as $m$-planes). We will always denote $N = m + n$. Moreover, we identify the spaces
\begin{equation}\label{iden}
\G = \{P \in \R^{N\times N}: P = P^t, \, P^2 = P, \, \tr(P) = m\}.
\end{equation}

\begin{Def}
An $m$-dimensional varifold $V$ in $\R^N$ is a Radon measure on $\R^{N}\times \G$. The varifold $V$ is said to be rectifiable if there exists an $m$-rectifiable set $\Gamma$ and an $\mathcal{H}^m\llcorner \Gamma$-measurable function $\theta: \Gamma \to (0, \infty)$ such that
\[
V(f) = \int_{\Gamma}f(x,T_x\Gamma)\theta(x)d\mathcal{H}^m(x), \qquad \forall f \in C_c(\R^N\times \G).
\]  
In this case, we denote $V = (\Gamma,\theta)$. If moreover $\theta$ takes values in $\mathbb N$, $V$ is said integer rectifiable. If $\theta = 1$ $\mathcal{H}^m\llcorner \Gamma$-a.e., then we will write $V = \llbracket \Gamma\rrbracket$.
\end{Def}

We will use $\|V\|$ to denote the projection in $\R^N$ of the measure $V$, i.e.
\[
\|V\|(A) := V(A\times \G), \qquad \forall A \subseteq \R^N, A \text{ Borel}.
\]
Hence  $\|V\|=p_\# V$, where \(p: \R^N\times \G\to \R^N\) is the projection onto the first factor and the push-forward  is intended in the sense of Radon measures.
\\
\\
Given an $m$-rectifiable varifold $V = (\Sigma,\theta)$ and $\Psi : \Sigma \to \R^N$ Lipschitz and proper (i.e. $\Psi^{-1}(K)$ is compact for every $K\subset \R^N$ compact), the image varifold of $V$ under $\psi$ is defined by
$$\psi^\#V := (\psi(\Sigma), \tilde \theta), \quad \mbox{where} \quad  \tilde \theta (y) := \sum_{x\in \Sigma \cap \psi^{-1}(y)}\theta(x).$$
Since $\psi$ is proper, we have that $\tilde \theta \mathcal{H}^m\llcorner \psi(\Sigma)$ is locally finite. By the area formula we get
$$\psi^\#V(f)=\int_{\psi(\Sigma)}f(x,T_x\Sigma)\tilde \theta(x)d\mathcal{H}^m(x)=\int_{\Sigma}f(\psi(x),d_x\psi(T_x\Sigma))J\psi(x,T_x\Sigma)\theta(x) d\mathcal{H}^m(x),$$
for every $f\in C^0_c(\R^N\times \G)$. Here $d_x\psi(S)$ is the image of $S$ under the linear map $d_x\psi(x)$ and 
\[
J\psi(x,S):=\sqrt{\det\Big(\big(d_x\psi\big|_S\big)^t\circ d_x\psi\big|_S\Big)}
 \]
 denotes the $m$-Jacobian determinant of the differential $d_x\psi$ restricted to the $m$-plane $S$, 
 see \cite[Chapter 8]{SimonLN}. Note that the push-forward of a varifold \(V\) is {\em not} the same as the push-forward of the  Radon measure \(V\) through a map $\psi$ defined on $\R^N\times \G$ (the latter being  denoted with an expressly different notation: \(\Psi_\# V\), see Section \ref{measures}).
\\
\\
Given $\Psi \in C^1(\G)$, we define the anisotropic energy on a rectifiable varifold $V=(\Gamma,\theta)$ as
\begin{equation}\label{envar}
\Sigma_\Psi(V):= \int_{\Gamma}\Psi(T_x\Gamma)\theta(z)d\mathcal{H}^m(z),\qquad \mbox{where } V = (\Gamma,\theta).
\end{equation}
We define the first variation of $V = (\Gamma,\theta)$ with respect to $\Psi$ as 
\[
[\delta_\Psi V](g) := \left.\frac{d}{d\eps}\right|_{\eps = 0}\Sigma_\Psi((\Phi_\eps)^\#(V)), \quad \forall g \in C^1_c(\R^N,\R^N),
\]
where $\Phi_\eps:=\id+\eps g$ is the flow generated by $g$. We rely on \cite[Lemma A.2]{DDG} to write the following expression for the variations:
\[
[\delta_\Psi V](g) = \int_{\R^N}\langle B_\Psi(T_x\Gamma),Dg(x)\rangle d\|V\|,
\]
where $B_\Psi(T)$ is defined through the equality
\begin{equation}\label{eq:B}
\langle B_\Psi(T), L\rangle := \Psi(T) \langle T,L\rangle+\big\langle D\Psi(T) ,\,T^\perp L T +(T^\perp L T)^t\big\rangle, \qquad \forall L\in \R^{N\times N}.
\end{equation}
Here, $D\Psi(T)$ denotes the differential of $\Psi$ once extended to a $C^1$ function defined in a small neighborhood of $\G$ in $\R^{N\times N}$. Let $U\subset \R^N$ open, we say that a varifold $V = (\Gamma,\theta)$ has \emph{$\Psi$-mean curvature in $L^p$ in $U$} if there exists a map $H \in L^p(\Gamma \cap U,\R^N; \mathcal{H}^m\llcorner \Gamma)$ such that
\begin{equation}\label{boundedmean}
[\delta_\Psi V](g) = -\int_{\R^N}(H(x),g(x))d\|V\|(x), \quad \forall g \in C^1_c(U,\R^N).
\end{equation}
If $H$ can be chosen to be $0$, we say that the varifold $V$ is \emph{stationary}.
\\
\\
The graph $\Gamma_u$ of a Lipschitz function $u : \Omega\subset \R^m \to \R^n$ defines an $m$-dimensional varifold with multiplicity 1, $\Gamma = \llbracket\Gamma_u\rrbracket$. Without loss of generality, we can suppose that the graph is parametrized on the first $m$ coordinates, so that $\Gamma_u := \{(x,y) \in \R^{m + n}: y = u(x)\}$. For notational purposes, we define the following maps:
\begin{equation}\label{maps}
M(X):=
\left(
\begin{array}{c}
id_m\\
X
\end{array}
\right), \quad \mbox{and} \quad \A(X) := \sqrt{\det(M(X)^tM(X))}, \quad \forall X \in \R^{n\times m},
\end{equation}
where $\A(X) $ simply corresponds to the area element of $X$, and
\begin{equation}\label{h}
h: \R^{n\times m} \to \R^{N\times N}, \quad h(X) := M(X)(M(X)^tM(X))^{-1}M(X)^t.
\end{equation}
Recalling \eqref{iden}, it is easily seen that $h(\R^{n\times m}) \subseteq \G$, and that $h$ is injective (it is, in fact, one of the \emph{canonical} charts of $\G$).
The map $h$ allows us to define the function $F_\Psi:\R^N \times \R^{n\times m} \to (0,+\infty)$ as:
\begin{equation}\label{functional}
F_\Psi(X) := \Psi(h(X))\A(X).
\end{equation}
through the identification \eqref{iden} and \eqref{h}.

\subsection{Young measures}

We refer the reader to \cite[Section 3]{DMU} for the results concerning Young measures that we are going to state without proof. Let $p > 1$ and let $v_j: \Omega \subseteq \R^m \to \R^d$ be a sequence of weakly convergent maps in $L^p(\Omega,\R^d)$ to an $L^p(\Omega,\R^d)$ map $v: \Omega \to \R^d$. The fundamental theorem on Young measures states that, up to considering a (non-relabeled) subsequence $v_j$, there exists a measurable map
\[
\Omega \ni x \mapsto \nu_x \in \mathcal M(\R^d), \quad \mbox{with} \quad \nu_x(\R^d)=1,
\]
such that, for every $f \in C(\R^{d})$ with
\[
|f(\Lambda)| \le C(1 + \|\Lambda\|^s), \qquad \forall \Lambda \in \R^{d},
\]
for some $s \in [1,p)$, it holds
\[
\int_{\Omega}f(v_j(x))\eta(x) dx \to \int_{\Omega}\left(\int_{\R^{d}}f(\Lambda)d\nu_x(\Lambda)\right)\eta(x)  dx, \qquad  \forall \eta \in C^\infty_{c}(\Omega).
\]
Moreover,
\begin{equation}\label{comp}
v_j \to v \text{ in measure} \quad \Leftrightarrow \quad  \nu_x = \delta_{v(x)}  \text{ for $\mathcal{H}^m$-a.e. $x \in \Omega$}.
\end{equation}
This in particular implies the strong convergence of $v_j$ to $v$ in $L^{q}(\Omega,\R^d)$, $\forall q \in [1,p)$.

\section{The atomic condition and related ellipticity conditions}\label{s:condition}

Let $\Psi: \mathbb{G}(N,m) \to (0,+\infty)$ be a $C^1$ function, and let $N = n + m$. As recalled in Subsection \ref{introvar}, the formula for the first variation of a varifold $V = (\Gamma,\theta)$ is given by
\begin{equation}\label{varB}
[\delta_\Psi T](g) = \int_{\Gamma}\langle B_\Psi(T_x\Gamma),Dg(x)\rangle \theta d\mathcal{H}^m(x).
\end{equation}
Recalling \eqref{eq:B}, we can readily compute the following expression of $B_\Psi$:
\[
B_\Psi(T) = \Psi(T)T + T^\perp d\Psi(T)T, \qquad \mbox{where }\quad d\Psi(T) := D\Psi(T) +D\Psi(T)^t.
\]
It is crucial to observe that, even if $D\Psi$ is the differential of $\Psi$ as a map defined on $\R^{N\times N}$, i.e. computed after extending $\Psi$ from $\G$ to a neighborhood of $\G$ in the whole space $\R^{N\times N}$, we have that $B_\Psi(T)$ does not depend on the particular chosen extension. This is true since, as computed in \cite[Lemma A.1]{DDG},
\[
T^\perp L T +(T^\perp L T)^t \in \Tan_T\G, \quad \forall L \in \R^{N\times N},
\]
and hence
\[
\big\langle D_T\Psi(x,T) ,\,T^\perp L T +(T^\perp L T)^t\big\rangle
\]
only depends on the values of $\Psi$ \emph{on} $\G$. In particular, $T^\perp d\Psi(T)T$ represents the differential of $\Psi$ on the manifold $\G$.
\\
\\
For every Borel  probability measure $\mu \in \mathcal M (\G)$, let us define  
\begin{equation}\label{eq:A}
A(\mu):=\int_{\G} B_\Psi(T)d\mu(T)\in\R^{N\times N}.
\end{equation}
\begin{Def}\label{atomica}
We say that \(\Psi\) satisfies the \emph{atomic condition} $(AC)$ if the following two conditions hold:
\begin{itemize}
\item[(AC1)]  \(\dim\Ker A(\mu)\le n\) for every probability measure \(\mu \in \mathcal M (\G)\),
\item[(AC2)]  if  \(\dim\Ker A(\mu)= n\), then  \(\mu=\delta_{T_0}\) for some \(T_0\in \G\).
\end{itemize}
\end{Def}

Aim of this section is to define two classes of integrands, see Definition \ref{DEFAC1} and \ref{DEFAC2}. We will show that the first class satisfies (AC1) in Proposition \ref{PROOFAC1} and that the second class satisfies (AC) in Proposition \ref{PROOFAC2}. The interest in these conditions is that they are \emph{open}, see Proposition \ref{OPENAC1} and \ref{OPENAC2}. Since the area functional satisfies the atomic condition, we will conclude that the integrands in a $C^1$ (resp. $C^2$) neighborhood of the area functional satisfy (AC1) (resp. (AC)), thus providing the first non-trivial class of integrands satisfying (AC1) or (AC) in general dimension and codimension. These results prove Theorems \ref{B} and \ref{C}.

\begin{Def}\label{DEFAC1}
We say that $\Psi: \mathbb{G}(N,m) \to \R$ satisfies the \emph{scalar (AC1) condition}, (SAC1), if there exists $\delta <\frac 1{m-1}$ such that
\begin{equation}\label{condition}
\langle B_\Psi(T)w,w\rangle\leq (1+\delta) \Psi(T) \|w\|^2, \qquad \forall T\in \mathbb{G}(N,m), w\in \R^N.
\end{equation}
\end{Def}

Before giving the second definition, let us introduce the following notation. Given a $C^1$ function $\Psi:\G \to \R$, we denote its dual function with $\Psi^*:\mathbb{G}(N,n) \to \R$, namely the function defined as
\[
\Psi^*(P) := \Psi(\id - P) = \Psi(P^\perp),\quad \forall P \in \mathbb{G}(N,n).
\]
Since one has
\[
B_\Psi(T) = \Psi(T)T + T^\perp d\Psi(T)T, \quad d\Psi(T) = D\Psi(T) +D\Psi(T)^t
\]
a simple computation shows that
\begin{equation}\label{BF*}
d\Psi^*(S^\perp)= -d\Psi(S)\quad \mbox{and hence} \quad B_{\Psi^*}(S^\perp) =  \Psi(S)S^\perp - S d\Psi(S)S^\perp,\quad \forall S \in \mathbb{G}(N,m),
\end{equation}
and hence elementary linear algebra gives us the following useful result:
\begin{equation}\label{zeroprod}
B_{\Psi}(S)^tB_{\Psi^*}(S^\perp) = 0,\quad \forall S \in \mathbb{G}(N,m).
\end{equation}

\begin{Def}\label{DEFAC2}
We say that $\Psi$ satisfies the \emph{scalar atomic condition} (SAC) if
\[
\langle B_\Psi(T), B_{\Psi^*}(S^\perp)\rangle > 0, \quad\forall T \neq S \in \mathbb{G}(N,m),
\]
and it satisfies the \emph{uniform scalar atomic condition} (USAC) if there exists a constant $C> 0$ independent of $T,S$ such that
\[
\langle B_\Psi(T), B_{\Psi^*}(S^\perp)\rangle \ge C\|T-S\|^2, \quad\forall T \neq S \in \mathbb{G}(N,m),
\]
\end{Def}

We will now show that if $\Psi$ satisfies (SAC1), then $\Psi$ fulfills (AC1).

\begin{prop}\label{PROOFAC1}
If $\Psi$ is a positive integrand satisfying (SAC1), then $\Psi$ satisfies (AC1).
\end{prop}
\begin{proof}
First, we observe that for every $T \in \mathbb{G}(N,m)$,
\begin{equation}\label{trace}
\tr (B_\Psi(T)) = m\Psi(T),
\end{equation}
or, in other words, $\tr (T^\perp d\Psi(T) T) = 0$. This can be seen immediately by the properties of $T \in \mathbb{G}(N,m)$. Now, assume by contradiction that there exists a probability measure $\mu\in \mathcal M(\mathbb{G}(N,m))$ such that $\dim(\ker (A(\mu))) =: d>n$. Since $\Psi$ is positive, then $\tr (A(\mu))>0$  by \eqref{trace}, hence $d<N$. Let $\{v_i\}_{i=1}^N$ be an orthonormal basis of $\R^N$ such that $\{v_i\}_{i=1}^d\subset \ker (A(\mu))$. Then, we get the following contradiction
\begin{equation*}
\begin{split}
m \int_{\mathbb{G}(N,m)}\Psi(T)d\mu(T) &\overset{\eqref{trace}}{=}\tr(A(\mu))= \sum_{i= 1}^{N}(A(\mu)v_i,v_i) = \sum_{i= d+1}^{N}(A(\mu)v_i,v_i)
\\
&=\sum_{i= d + 1}^{N}\int_{\mathbb{G}(N,m)}(B_\Psi(T)v_i,v_i)d\mu(T)\overset{\eqref{condition}}{\leq} \sum_{i=d + 1}^{N}\int_{\mathbb{G}(N,m)}(1+\delta) \Psi(T) d\mu(T)\\
&= (N - d - 1) (1+\delta)\int_{\mathbb{G}(N,m)}\Psi(T)d\mu(T)< m \int_{\mathbb{G}(N,m)}\Psi(T)d\mu(T),
\end{split}
\end{equation*}
the last inequality being true due to the non-negativity of $\Psi$ and the estimate
\[
(N-d-1)(1 + \delta) < (N-n-1) \frac{m}{m -1} = m.
\]
\end{proof}

Now let us turn to (SAC).

\begin{prop}\label{PROOFAC2}
If $\Psi$ is a positive integrand satisfying (SAC), then $\Psi$ satisfies (AC).
\end{prop}
\begin{proof}
Suppose 
\begin{equation}\label{kergrande}
\dim(\Ker A(\mu)) \ge n
\end{equation}
for some probability measure $\mu\in \mathcal M(\mathbb{G}(N,m))$. We need to show that in fact $\mu = \delta_{T_0}$ for some $T_0 \in G(N,n)$. Once this is established, it follows that also $\dim(\Ker A(\mu)) = n$. Indeed, in the case $\mu = \delta_{T_0}$, we have
\[
0 = A(\mu)w = B_\Psi(T_0)w = \Psi(T_0)T_0w + T_0^\perp d\Psi(T_0)T_0w,
\]
and this can happen if and only if $\Psi(T_0)T_0w = 0$ and $T_0^\perp d\Psi(T_0)T_0w = 0$. The sign assumption on $\Psi$ therefore would yield
\[
\Ker(B_\Psi(T_0)) = \Ker(T_0),
\]
and hence that $B_\Psi(T_0)=A(\mu)$ has $n$-dimensional kernel.

We are left to show that \eqref{kergrande} implies $\mu = \delta_{T_0}$ for some $T_0 \in G(N,n)$. By \eqref{kergrande} we find an orthonormal system $v_1,\dots, v_n$ inside $\Ker A(\mu)$, and define $P \in \mathbb{G}(N,n)$ to be the orthogonal projection onto it. First, by \eqref{BF*}, we have $B_{\Psi^*}(P)w = 0$ for every $w \perp \spn\{v_1,\dots, v_n\}$. Therefore, since $v_i \in \Ker (A(\mu))$ for every $i=1,\dots, n$, we see that
\[
\int_{\mathbb{G}(N,m)}\langle B_\Psi(T),B_{\Psi^*}(P)\rangle d\mu = \langle A(\mu), B_{\Psi^*}(P)\rangle = \sum_{i = 1}^n (A(\mu)v_i,B_{\Psi^*}(P)v_i) = 0.
\]
However, by (SAC)
\[
0 = \int_{\mathbb{G}(N,m)}\langle B_\Psi(T),B_{\Psi^*}(P)\rangle d\mu \ge 0,
\]
with equality if and only if $\mu$ is concentrated on $T_0 = P^\perp$, and this finishes the proof.
\end{proof}

Let us comment on the necessity of the sign assumption on $\Psi$ in Proposition \ref{PROOFAC2}:

\begin{lemma}\label{positive} 
Let $\Psi \in C^1(\G)$ satisfy (AC). Then $\Psi$ is either nonnegative or nonpositive.
\end{lemma}

\begin{proof}
The key idea is to restrict $\Psi$ to the space of codimension one planes. Let $T_0 \in \G$ be arbitrary but fixed. Suppose $T_0$ is the projection on $\spn\{v_1,\dots, v_m\}$, for an orthonormal system $v_1,\dots, v_m$. We want to show that if $T_1 \in \G$ is the projection on $\spn\{w,v_2,\dots, v_m\}$, for some $w \in \{v_2,\dots, v_m\}^\perp$ then
\begin{equation}\label{signprod}
\Psi(T_0)\Psi(T_1) \ge 0.
\end{equation}
If we show this, then we readily conclude the Lemma by iterating this claim. Hence fix $w \in \{v_2,\dots, v_m\}^\perp$. Without loss of generality, we can assume that $\{e_1,\dots, e_{m+1}\}$ is an orthonormal system spanning $\{w,v_1,v_2,\dots, v_m\}$, where $\{e_1,\dots, e_{N}\}$ is the canonical orthonormal basis for $\R^N$. We want to define $\Phi \in C^1(\mathbb{G}(m+1,m))$ as a \emph{restriction} of $\Psi$. To this aim we consider any $m$-plane $\pi$ with
\[
\pi \subset \{e_1,\dots, e_{m+1}\} \sim \R^{m + 1},
\]
we denote $P\in \R^{(m + 1)\times (m + 1)}$ to be the projection onto $\pi$ and simply set $\tilde P= M P M^t \in \R^{N\times N}$, where $M=(e_1|\dots |e_{m+1})$. We define $$\Phi( P) := \Psi(\tilde P).$$ Now we show that $\Phi$ still fulfills (AC). 
To this aim, 
we simply observe that, by definition of $\Phi$, $B_\Phi(P)$ is the $\R^{(m + 1)\times (m + 1)}$ matrix obtained by
\[
B_\Phi(P)=\Phi(P)P + P^\perp d\Phi(P)P, \quad d\Phi(P) = D\Phi(P) +D\Phi(P)^t.
\]
By the chain rule we deduce that
\[
d\Phi(P)=M^td\Psi(\tilde P)M.
\]
This in turn implies $B_\Phi(P)=M^t B_\Psi(\tilde P)M$ and, since $\Psi$ satisfies (AC), we conclude that $\Phi$ fulfills (AC) as well.
If we prove that $\Phi$ has a sign, then \eqref{signprod} readily follows. 

In the case $m>1$, we have the following easy contradiction argument. Assume by contradiction that there exists $T_1\in \mathbb{G}(m + 1,m)$ such that $\Phi(T_1)=0$. 
By definition of $B_\Phi(T_1)$
\begin{equation}\label{intermm}
B_\Phi(T_1)=T_1^\perp d\Phi(T_1)T_1, \quad d\Phi(T_1) = D\Phi(T_1) +D\Phi(T_1)^t.
\end{equation}
Since $\Phi$ fulfills (AC), then $\dim(\Ker B_\Phi(T_1))=1$ and $\dim(\mbox{Im} B_\Phi(T_1))=m$. On the other hand, \eqref{intermm} implies that $\mbox{Im} B_\Phi(T_1)\subset \mbox{Im} T_1^\perp$ and consequently we get the following contradiciton
$$
m=\dim(\mbox{Im} B_\Phi(T_1))\leq \dim(\mbox{Im} T_1^\perp)=1.
$$

In the general case, including $m=1$, we can conclude with the following argument, which borrows ideas from \cite[Theorem 1.3]{DDG}.
We define \(G\) as the one-homogeneous extension to $\R^{m+1}$ of the following map $\mathbb S^{m} \ni \nu  \mapsto \Phi(\id - \nu\otimes \nu)\). Since $\Phi$ satisfies (AC), we deduce, as in Step 2 of the proof of \cite[Theorem 1.3]{DDG}, that for every \(\nu\neq \pm\bar \nu \in\mathbb S^{m}\):
\begin{equation*}
G(\nu)G(\bar \nu)-\langle d_\nu G(\bar \nu),\nu\rangle \langle d_\nu G( \nu),\bar \nu\rangle \ne 0.
\end{equation*}
Applying \cite[Lemma 1]{SAP} with $K=\mathbb G(m+1,m)$, we deduce that either
\begin{equation}\label{e:bella}
G(\nu)G(\bar \nu)-\langle d_\nu G(\bar \nu),\nu\rangle \langle d_\nu G( \nu),\bar \nu\rangle > 0\qquad\mbox{for all $\nu,\bar\nu$ s.t. \(\nu\ne \pm \bar\nu\),}
\end{equation}
or that \eqref{e:bella} holds with the opposite sign. Without loss of generality we will treat the positive case \eqref{e:bella}. If $G(\bar \nu)\neq 0$, we conclude that $G(\nu)>\langle d_\nu G(\bar \nu),\nu\rangle$  for every $\nu\neq  \pm \bar\nu \in \mathbb S^m$ as in \cite{DDG}. Since $\Phi$ satisfies (AC), it is easy to see that the zero set of $G$ cannot contain open subsets and, by continuity of $G$, we conclude that $G(\nu)\geq \langle d_\nu G(\bar \nu),\nu\rangle$ for every $\nu,\bar \nu\in \mathbb S^m$. In other words, $G$ is an even, convex and one-homogeneous function on $\R^{m + 1}$. We claim that this implies $G$ (and hence $\Phi$) is nonnegative. Indeed, assume by contradiction there exists $\nu\in  \mathbb S^m$ such that $G(\nu)<0$, then being $G$ even, $G(-\nu)<0$. Hence, by convexity of $G$, we deduce $G(0)<0$, which contradicts the one-homogeneity of $G$.
\end{proof}

\begin{remark}
Notwithstanding we focus on positive integrands, in view of Lemma \ref{positive}, it is clear that all the results of this Section hold for negative integrands as well, provided to change the sign in (SAC) and (USAC).
\end{remark}

Finally, we show that Definitions \ref{DEFAC1} and \ref{DEFAC2} are open conditions.

\begin{prop}\label{OPENAC1}
Let $\Psi\in C^1(\G,(0,\infty))$ satisfy (SAC1). Then, there exists $\eps = \eps(\delta,n,m, \min_{T} \Psi(T))>0$ such that if $\Psi': \G \to \R^+$ satisfies
\begin{equation}\label{closeC1}
\|\Psi - \Psi'\|_{C^1(\G)} \le \eps,
\end{equation}
then $\Psi'$ also satisfies (SAC1). Here $\delta>0$ is the quantity appearing in Definition \ref{DEFAC1} and depends only on $\Psi$.
\end{prop}
\begin{proof}
As observed at the beginning of the section, the term $T^\perp d\Psi(T)T$ is the differential of $\Psi$ on $\G$. Thus, we see that
\begin{equation}\label{continuity}
\|B_\Psi - B_{\Psi'}\|_{C^0(\G)} \le c\|\Psi-\Psi'\|_{C^1(\G)}, 
\end{equation}
where $c>0$ is a dimensional constant. We assume $\eps$ is chosen so small that also $\Psi'$ fulfilling \eqref{closeC1} is strictly positive and actually we enforce
\begin{equation}\label{min'}
\min_{T \in \G}\Psi'(T) \ge \gamma := \frac{1}{2}\min_{T \in \G}\Psi(T)>0.
\end{equation}
The result now easily follows, in fact fix $\delta <\frac{1}{m-1}$ such that
\[
(B_\Psi(T)w,w) \le (1+\delta)\Psi(T), \quad \forall T \in \G, \, w \in \mathbb{S}^{N - 1}.
\]
Then, for each $T \in \G$ and $w \in \mathbb{S}^{N - 1}$
\begin{align*}
(B_{\Psi'}(T)w,w) &= ((B_{\Psi'}(T)-B_{\Psi}(T))w,w) + (B_\Psi(T)w,w) \overset{\eqref{continuity}}{\le} c\|\Psi-\Psi'\|_{C^1(\G)} + (1 + \delta)\Psi(T) \\
& =  c\|\Psi-\Psi'\|_{C^1(\G)} + (1 + \delta)(\Psi(T)-\Psi'(T)) + (1 + \delta)\Psi'(T)\\
& \overset{\eqref{min'}}{\le} \frac{(c + 1 + \delta)}{\gamma}\eps\Psi'(T) + (1 + \delta)\Psi'(T).
\end{align*}
Choosing $\eps>0$ sufficiently small, we can still enforce $\frac{c + 1 + \delta}{\gamma}\eps + \delta < \frac{1}{m - 1}$ and hence conclude the proof.
\end{proof}

\begin{remark}\label{remark:AC1}
As written in the introduction, Proposition \ref{OPENAC1} yields an interesting immediate corollary: in codimension one, the validity of (AC1) for $\Psi \in \mathbb{G}(m + 1,m)$ does not imply (in general) the convexity of the one-homogeneous extension to $\R^{m+1}$ of $\Psi^* \in \mathbb{G}(m+1,1)$. In fact, it is straightforward to see that arbitrarily small $C^1$ perturbations of convex functions need not to be convex. We recall that (AC1) implies the rectifiability of the mass of a varifold with locally bounded first variation,  \cite[Page 656, point (b)]{RDHR}. This remark shows that the result in \cite[Page 656, point (b)]{RDHR} is indeed optimal.
\end{remark}

\begin{prop}\label{OPENAC2}
Let $\Psi\in C^2(\G,(0,\infty))$ satisfy (USAC). Then, there exists $\eps = \eps(C, \|\Psi\|_{C^2},n,m)>0$ such that if $\Psi': \G \to \R^+$ satisfies
\[
\|\Psi - \Psi'\|_{C^2(\G)} \le \eps,
\]
then $\Psi'$ also satisfies (USAC). Here $C>0$ is the quantity appearing in Definition \ref{DEFAC2} and only depends on $\Psi$.
\end{prop}
\begin{proof}
For any $\Phi \in C^2(\G)$, let us denote with $A_\Phi(T) := T^\perp d\Phi(T) T$ and with $\Delta(T) := \Psi'(T)-\Psi(T)$. We have
\begin{equation}\label{00}
\begin{split}
\langle B_\Phi(T), B_{\Phi^*}(S^\perp)\rangle &= \langle \Phi(T)T + T^\perp d\Phi(T)T, \Phi(S)S^\perp - Sd\Phi(S)S^\perp\rangle\\
&= \langle \Phi(T)T + A_\Phi(T), \Phi(S)S^\perp - A_\Phi(S)^t\rangle\\
&=\Phi(T)\Phi(S)\langle T,S^\perp\rangle + \Phi(S)\langle A_\Phi(T),S^\perp\rangle\\
&\qquad-\Phi(T)\langle T, A_\Phi(S)^t\rangle - \langle A_\Phi(T), A_\Phi(S)^t\rangle.
\end{split}
\end{equation}
We rewrite separately these terms, with the aim of comparing
\[
\langle B_\Psi(T), B_{\Psi^*}(S^\perp)\rangle \text{ and } \langle B_{\Psi'}(T), B_{(\Psi')^*}(S^\perp)\rangle.
\]
Firstly, we have:
\begin{equation}\label{11}
\Psi'(S)\Psi'(T)\langle T,S^\perp\rangle = \Psi(T)\Psi(S)\langle T,S^\perp\rangle+ \Delta(T)\Psi(S)\langle T,S^\perp\rangle + \Psi'(T)\Delta(S)\langle T,S^\perp\rangle.
\end{equation}
Secondly, we write
\begin{equation}\label{22}
\begin{split}
 \Psi'(S)\langle &A_{\Psi'}(T),S^\perp\rangle -\Psi'(T)\langle T, A_{\Psi'}(S)^t\rangle \\
&= (\Psi'(S)-\Psi'(T))\langle A_{\Psi'}(T),S^\perp\rangle -\Psi'(T)(\langle T, A_{\Psi'}(S)^t\rangle-\langle A_{\Psi'}(T),S^\perp\rangle)\\
&= (\Delta(S)-\Delta(T))\langle A_{\Psi'}(T),S^\perp\rangle +  (\Psi(S)-\Psi(T))\langle A_{\Psi'}(T),S^\perp\rangle \\
&\qquad-\Delta(T)(\langle T, A_{\Psi'}(S)^t\rangle-\langle A_{\Psi'}(T),S^\perp\rangle)-\Psi(T)(\langle T, A_{\Psi'}(S)^t\rangle-\langle A_{\Psi'}(T),S^\perp\rangle)\\
&= (\Delta(S)-\Delta(T))\langle A_{\Psi'}(T),S^\perp\rangle  +(\Psi(S)-\Psi(T))\langle A_{\Psi'-\Psi}(T),S^\perp\rangle \\
&\qquad+ (\Psi(S)-\Psi(T))\langle A_{\Psi}(T),S^\perp\rangle -\Delta(T)(\langle T, A_{\Psi'}(S)^t\rangle-\langle A_{\Psi'}(T),S^\perp\rangle)\\
&\qquad-\Psi(T)(\langle T, A_{\Psi'-\Psi}(S)^t\rangle-\langle A_{\Psi'-\Psi}(T),S^\perp\rangle)-\Psi(T)(\langle T, A_{\Psi}(S)^t\rangle-\langle A_{\Psi}(T),S^\perp\rangle)\\
&= \Psi(S)\langle A_{\Psi}(T),S^\perp\rangle -\Psi(T)\langle T, A_{\Psi}(S)^t\rangle\\
& \qquad + (\Delta(S)-\Delta(T))\langle A_{\Psi'}(T),S^\perp\rangle +(\Psi(S)-\Psi(T))\langle A_{\Psi'-\Psi}(T),S^\perp\rangle \\
&\qquad-\Delta(T)(\langle T, A_{\Psi'}(S)^t\rangle-\langle A_{\Psi'}(T),S^\perp\rangle)-\Psi(T)(\langle T, A_{\Psi'-\Psi}(S)^t\rangle-\langle A_{\Psi'-\Psi}(T),S^\perp\rangle).
\end{split}
\end{equation}
Thirdly, and finally, we rewrite
\begin{equation}\label{33}
 \langle A_{\Psi'}(T), A_{\Psi'}(S)^t\rangle =\langle A_{\Psi}(T), A_{\Psi}(S)^t\rangle+  \langle A_{\Psi'-\Psi}(T), A_{\Psi'}(S)^t\rangle +\langle A_{\Psi}(T), A_{\Psi'-\Psi}(S)^t\rangle.
\end{equation}
Combining \eqref{00}-\eqref{11}-\eqref{22}-\eqref{33}, we find that
\begin{align}
\langle B_{\Psi'}(T),& B_{(\Psi')^*}(S^\perp)\rangle - \langle B_{\Psi}(T), B_{\Psi^*}(S^\perp)\rangle\label{44} \\
\label{firstline} & =\Delta(T)\Psi(S)\langle T,S^\perp\rangle + \Psi'(T)\Delta(S)\langle T,S^\perp\rangle\\
& + (\Delta(S)-\Delta(T))\langle A_{\Psi'}(T),S^\perp\rangle +(\Psi(S)-\Psi(T))\langle A_{\Psi'-\Psi}(T),S^\perp\rangle\label{secondline} \\
&-\Delta(T)(\langle T, A_{\Psi'}(S)^t\rangle-\langle A_{\Psi'}(T),S^\perp\rangle)-\Psi(T)(\langle T, A_{\Psi'-\Psi}(S)^t\rangle-\langle A_{\Psi'-\Psi}(T),S^\perp\rangle)\label{thirdline}\\
& -\langle A_{\Psi'-\Psi}(T), A_{\Psi'}(S)^t\rangle -\langle A_{\Psi}(T), A_{\Psi'-\Psi}(S)^t\rangle\label{fourthline}.
\end{align}
In order to conclude the proof, we only need to show that this right hand side is bounded (in absolute value) by
\begin{equation}\label{secorder}
c\|\Psi-\Psi'\|_{C^2}\|T-S\|^2,
\end{equation}
for a constant $c = c(n,m, \|\Psi\|_{C^2})$. Once this is done, we can choose $\eps$ sufficiently small to ensure that also $\Psi'$ fulfills (USAC). We will assume without loss of generality that $\|\Psi-\Psi'\|_{C^2} \le 1$. We show separately that a bound of the form \eqref{secorder} holds for \eqref{firstline}-\eqref{secondline}-\eqref{thirdline}-\eqref{fourthline}.
\\
\\
\indent\fbox{Estimate of \eqref{firstline}:}
\\
\\
Clearly,
\[
|\Delta(T)\Psi(S)\langle T,S^\perp\rangle + \Psi'(T)\Delta(S)\langle T,S^\perp\rangle| \le (|\Delta(T)|+|\Delta(S)|)(\|\Psi\|_\infty + \|\Psi'\|_\infty)|\langle T,S^\perp\rangle|.
\]
Moreover,
\begin{equation}\label{idTS}
\langle T,S^\perp\rangle = \langle T, \id\rangle - \langle T,S\rangle = m - \langle T,S\rangle =\frac{1}{2}\|T-S\|^2.
\end{equation}
Since clearly $|\Delta(T)| \le \|\Psi-\Psi'\|_{C^2}$ and we assumed that $\|\Psi-\Psi'\|_{C^2} \le 1$, we find that for the term \eqref{firstline} an estimate as \eqref{secorder} holds. 
\\
\\
\indent\fbox{Estimate of \eqref{secondline}:}
\\
\\
Here, the estimate \eqref{secorder} follows in a similar fashion for all the addenda. Indeed, it follows from the estimates
\begin{align*}
&|\Delta(S)-\Delta(T)| \le c'(n,m) \|\Psi-\Psi'\|_{C^2}\|T-S\|,\\
&\|A_{\Psi'-\Psi}(T)\| \le \|\Psi-\Psi'\|_{C^2},\\
&|\Psi(S)-\Psi(T)| \le \|\Psi\|_{C^2}\|T-S\|,
\end{align*}
and the crucial observation that for any $\Phi \in C^2(\G)$, we have
\begin{equation}\label{zerest}
|\langle A_\Phi(T),S^\perp\rangle| \le \|\Phi\|_{C^2}\|T-S\|.
\end{equation}
The latter holds true since
\begin{equation}\label{zero}
\langle A_\Phi(T),S^\perp\rangle = \langle A_\Phi(T),S^\perp-T^\perp\rangle,
\end{equation}
as can be seen by the simple computation:
\[
\langle A_\Phi(T),T^\perp\rangle = \langle T^\perp d\Phi(T)T,T^\perp\rangle = \langle T^\perp d\Phi(T)TT^\perp,\id\rangle =0,
\]
since $TT^\perp = 0, \forall T \in \G$. Therefore, for \eqref{secondline}, an estimate like \eqref{secorder} is true. 
\\
\\
\indent\fbox{Estimate of \eqref{thirdline}:}
\\
\\
All of the terms in \eqref{thirdline} contain a \emph{commutator} of the form
\[
\langle T, A_{\Phi}(S)^t\rangle-\langle A_{\Phi}(T),S^\perp\rangle.
\]
If we see that for the latter it holds
\begin{equation}\label{secordergen}
|\langle T, A_{\Phi}(S)^t\rangle-\langle A_{\Phi}(T),S^\perp\rangle| \le c'\|\Phi\|_{C^2}\|T-S\|^2,
\end{equation}
for some constant $c' = c'(n,m)$, then we see easily that for the terms in \eqref{thirdline} an estimate of the form \eqref{secorder} is fulfilled.  To see that \eqref{secordergen} is true, write
\begin{equation}\label{comm}
\begin{split}
\langle T, A_{\Phi}(S)^t\rangle-\langle A_{\Phi}(T),S^\perp\rangle &= \langle T, A_{\Phi}(S)\rangle-\langle A_{\Phi}(T),S^\perp\rangle = \langle T-S, A_{\Phi}(S)\rangle-\langle A_{\Phi}(T),S^\perp-T^{\perp}\rangle\\
&= \langle T-S, A_{\Phi}(S)\rangle+\langle A_{\Phi}(T),S-T\rangle = -\langle A_{\Phi}(T) - A_{\Phi}(S),T-S\rangle,
\end{split}
\end{equation}
the second equality being true by \eqref{zero} and the analogous identity $\langle A_\Phi(S),S\rangle=0$, for every  $S \in \G$. Now the previous inequality shows that \eqref{secordergen} holds and hence also the proof of this estimate is finished.
\\
\\
\indent\fbox{Estimate of \eqref{fourthline}:}
\\
\\
The two addenda in \eqref{fourthline} are estimated in a similar fashion, so let us consider only the first. Using the properties of the projection matrices $T,S,T^\perp, S^\perp$:
\begin{align*}
\langle &A_{\Psi'-\Psi}(T), A_{\Psi'}(S)^t\rangle = \langle T^\perp (d\Psi'-d\Psi)(T)T,Sd\Psi'(S)S^\perp\rangle = \langle T^\perp T^\perp (d\Psi'-d\Psi)(T)TT,SSd\Psi'(S)S^\perp S^\perp\rangle \\
&= \langle T^\perp (d\Psi'-d\Psi)(T)T,(T^\perp S)Sd\Psi'(S)S^\perp (S^\perp T)\rangle = \langle A_{\Psi'-\Psi}(T), (T^\perp S)A_{\Psi'}(S)^t(S^\perp T)\rangle.
\end{align*}
Since
\[
T^\perp S = T^{\perp}(S - T) \text{ and } S^\perp T = S^\perp (T-S),
\]
we readily see that
\[
|\langle A_{\Psi'-\Psi}(T), A_{\Psi'}(S)^t\rangle| \le c\|\Psi-\Psi'\|_{C^2}\|T-S\|^2,
\]
for some $c = c(n,m,\|\Psi\|_{C^2})>0$, and therefore also \eqref{fourthline} fulfills an estimate as in \eqref{secorder}.
\\
\\
From these computations, we deduce that also \eqref{44} satisfies an estimate as in \eqref{secorder}, i.e. 
\[
|\langle B_{\Psi'}(T), B_{(\Psi')^*}(S^\perp)\rangle - \langle B_{\Psi}(T), B_{\Psi^*}(S^\perp)\rangle| \le c\|\Psi-\Psi'\|_{C^2}\|T-S\|^2,
\]
and hence we conclude the proof of the Proposition.
\end{proof}

\begin{remark}
The reader may wonder whether the use of the $C^2$ norm in Proposition \ref{OPENAC2} can be relaxed to the $C^1$ norm. This is not merely a technical difficulty. Indeed, in the codimension one case it is known by the aforementioned \cite[Theorem 1.3]{DDG} that $\Psi \in \mathbb{G}(m + 1,m)$ satisfies (AC) if and only if $\Psi^* \in \mathbb{G}(m+1,1)$ is a strictly convex function, once we consider its one-homogeneous extension to $\R^{m + 1}$. As in Remark \ref{remark:AC1}, we observe that a $C^1$ perturbation of a strictly convex function is not, in general, a strictly convex function, hence (AC) cannot have open interior with respect to the $C^1(\G)$-topology on integrands.
\end{remark}

\begin{Cor}
For any $m,n > 0$, (AC1) holds in a $C^1$ neighborhood of the area functional and (AC) holds in a $C^2$ neighborhood of the area functional on $\G$.
\end{Cor}
\begin{proof}
Let us denote the area functional as $\Psi \equiv 1$. Therefore,
\[
B_\Psi(T) = T.
\]
$\Psi$ satisfies (SAC1), since
\[
(Tw,w) \le 1, \qquad \forall w \in \mathbb{S}^{m - 1}.
\]
Moreover,
\[
\langle B_\Psi(T),B_{\Psi*}(S^\perp)\rangle = \langle T,S^\perp\rangle \overset{\eqref{idTS}}{=} \frac{1}{2}\|T-S\|^2,
\]
and hence the area functional fulfills (USAC) with constant $C = \frac{1}{2}$. Now Propositions \ref{PROOFAC1} \& \ref{OPENAC1} and \ref{PROOFAC2} \& \ref{OPENAC2} conclude the proof, once we observe that if we take the $C^1$ (or $C^2$) norm sufficiently small, also the perturbed functional will be positive.
\end{proof}

In the next section we will need the following observation, namely that (USAC) implies quasiconvexity:
\begin{prop}
Let $\Psi\in C^2(\G,(0,\infty))$ be a functional that satisfies (USAC) with constant $C > 0$ and $F_\Psi$ as in \eqref{functional}.
Then $F_\Psi$ is quasiconvex in the following sense. There exists $\alpha(n,m,C,\min \Psi)>0$ such that, for every  open, bounded $\Omega\subset \R^m$ and for every $A\in \R^{n\times m}$ it holds
\begin{equation}\label{quasiconvex}
\int_\Omega F_\Psi(A+D\varphi)-F_\Psi(A)dx \geq \alpha \int_\Omega \A(A+D\varphi)-\A(A)dx, \qquad \forall \varphi \in C^1_c(\Omega, \R^n).
\end{equation}
In particular $F_\Psi$ satisfies the following local uniform Legendre-Hadamard condition, namely
\begin{equation}\label{rank1}
D^2F_\Psi(X)[M,M]\geq c\|M\|^2, \qquad \mbox{for every $X\in B_R$ and $M\in \R^{n\times m}$ with $\rank(M)=1$}.
\end{equation}
\end{prop}
\begin{proof}
By Proposition \ref{OPENAC2}, there exists $\alpha(n,m,C,\min \Psi)>0$ small enough such that $\Psi':=\Psi-\alpha$ also satisfies (USAC) and is positive. In particular, by Proposition \ref{PROOFAC2}, we deduce that $\Psi'$ satisfies (AC) and, by \cite[Theorem A, Theorem 8.8]{DRK}, $\Psi'$ satisfies the strict Almgren ellipticity condition, \cite[Definition 4.5]{DRK}. In particular we conclude that 
$$
\int_{\Gamma_{A+\varphi}}\Psi'(T_y \Gamma_{A+\varphi}) d\mathcal{H}^m-\int_{ \Gamma_{A}}\Psi'(T_y\Gamma_{A}) d\mathcal{H}^m>0,
$$
which in turn reads
$$
\int_{\Gamma_{A+\varphi}}\Psi(T_y \Gamma_{A+\varphi}) d\mathcal{H}^m-\int_{ \Gamma_{A}}\Psi(T_y\Gamma_{A}) d\mathcal{H}^m> \alpha(\mathcal{H}^m(\Gamma_{A+\varphi})-\mathcal{H}^m( \Gamma_{A})).
$$
Hence, by the area formula and \eqref{functional}, we conclude \eqref{quasiconvex}.
Moreover, by \cite[Lemma 4.3]{DMU} applied to $F_\Psi-\alpha \A$, we deduce that $F_\Psi-\alpha \A$ is rank-one convex, i.e. $D^2(F_\Psi-\alpha\A)(X)[M,M]\geq 0$ for every $X\in \R^{n\times m}$ and $M\in \R^{n\times m}$ with $\rank(M)=1$. 
One can check that $\A$ satisfies the local uniform Legendre-Hadamard condition, see \cite[Lemma 6.5]{TR}.  \cite[Lemma 6.5]{TR} is proved for $m=2$, but it readily extends to every $m$. Then we deduce the validity of \eqref{rank1}.
\end{proof}

\section{Regularity}\label{s:regularity}

In this section, we want to prove the following:

\begin{Teo}\label{regteo}
Let $\Psi\in C^2(\G,(0,\infty))$ be a functional satisfying (USAC), let $p>m$ and consider an open, bounded set $\Omega\subset \R^m$. Let $u \in \Lip(\Omega,\R^n)$ be a map whose graph $\Gamma_u$ induces a varifold with $\Psi$-mean curvature $H$ in $L^p$ in $\Omega \times \R^n$. Then there exists $\alpha>0$ and an open set $\Omega_0$ of full measure in $\Omega$ such that
\[
u \in C^{1,\alpha}(\Omega_0,\R^n).
\]
\end{Teo}

\begin{remark}\label{Remark:nonaut}
As mentioned in the introduction, without loss of generality, in this paper we treat autonomous integrands as in Theorem \ref{regteo}. Nevertheless, we remark that Theorem \ref{regteo} can be easily extended to non autonomous integrands $\Psi\in C^2(\R^N\times\G,(0,\infty))$ satisfying (USAC) at every $x\in \R^N$. Requiring (USAC) at every $x\in \R^N$ means that there exists a constant $C> 0$ independent of $T,S,x$ such that for every $x\in \R^N$
\[
\langle B_\Psi(x,T), B_{\Psi^*}(x,S^\perp)\rangle \ge C\|T-S\|^2, \quad\forall T \neq S \in \mathbb{G}(N,m),
\]
where $B_\Psi(x,T)$ is defined in \cite[Equation (2.6)]{DDG}.
Indeed, in this case the first variation for a rectifiable varifold $V=(\Gamma,\theta)$ yields 
\begin{equation}\label{curva}
\int_{U}\langle d_y\Psi(y,T_y\Gamma),g\rangle + \langle B_\Psi(y,T_y\Gamma),Dg\rangle d\|V\|(y) = -\int_U (H,g)d\|V\|(y).
\end{equation}
We can absorb the term $\langle d_y\Psi(y,T_y\Gamma),g\rangle$ into the right hand side, obtaining an equation similar to \eqref{curv}. We can consequently carry the same regularity analysis.
\end{remark}

The proof goes as follows. We prove, in Proposition \ref{Cac} a Caccioppoli inequality similar to the one obtained by Allard in the case of general varifolds with bounded mean curvature. The latter is the main novelty of our approach, and uses essentially the (USAC) property of $\Psi$. Subsequently, in Corollary \ref{graphCac} we will show how to get  a Caccioppoli inequality for $u$ as in Theorem \ref{regteo} from Proposition \ref{Cac}. In Proposition \ref{Tilt}, we will see how this Caccioppoli inequality implies a so-called \emph{decay of the excess}, analogous to \cite[Lemma 4.1]{EVA}. From that point on, the proof becomes rather standard, so we will only sketch how to conclude the proof of Theorem \ref{regteo}. The interested reader may consult \cite[Section 7]{EVA}.

\begin{prop}[Caccioppoli inequality]\label{Cac}
Let $\Psi$ as in Theorem \ref{regteo}. Let $V = \llbracket \Gamma,\theta\rrbracket$ be a rectifiable varifold with $\Psi$-mean curvature $H$ bounded in $L^2$ in $U \subset \R^{N}$. Then, there exists $C_2 = C_2(n,m,\|\Psi\|_{C^2},C) > 0$, where $C$ is the constant in Definition \ref{DEFAC2} for $\Psi$, such that:
\begin{equation}\label{Caceq}
\fint_{B_{r}(z)}\|T_y\Gamma - S\|^2d\|V\|(y) \le \frac{C_2}{r^2}\fint_{B_{2r}(z)}\dist(y-p,\pi)^2d\|V\|(y) + r^2C_2\fint_{B_{2r}(z)}\|H\|^2d\|V\|(y),
\end{equation}
for every $p\in \R^N$, $S \in \G, z \in U, r >0$ such that $\dist(z,\partial U) \le 4r$, where $\pi=\text{\em Im}(S)$.
\end{prop}

\begin{proof}
In this proof we will denote with $C_2$ a positive constant which may change line by line, but that shall always depend just on $n,m,\|\Psi\|_{C^2},C$.
Using the definition of $\Psi$-mean curvature for varifolds with bounded anisotropic first variation, we have that for every $g \in C^\infty_c(U,\R^N)$,
\begin{equation}\label{curv}
\int_{U}\langle B_\Psi(T_y\Gamma),Dg\rangle d\|V\|(y) = -\int_U (H,g)d\|V\|(y).
\end{equation}
We fix $S \in \G$, and we prove the assertion in the case $r = 1, z = 0$, the general case being true by scaling and translating. Under these assumptions, choose
\[
g(y):= \varphi^2(y) \BBB (y-p),
\]
for a radial $\varphi \in C^\infty_c(\R^N)$, with $\varphi \in [0,1]$, $\varphi \equiv 1$ on $B_1(0)$ and $\varphi \equiv 0$ on $B_2(0)^c$. With this choice of $g$, \eqref{curv} reads:
\begin{align*}
\int_{U}\varphi^2(y)\langle B_\Psi(T_y\Gamma),\BBB\rangle d\|V\|(y) &+ 2\int_{U} \varphi(y)( B_\Psi(T_y\Gamma)D\varphi(y),\BBB (y-p)) d\|V\|(y) \\
&= -\int_U \varphi^2(y)(H,\BBB (y-p))d\|V\|(y),
\end{align*}
that we rewrite as
\begin{align}
&\int_{U}\varphi^2(y)\langle B_\Psi(T_y\Gamma),\BBB\rangle d\|V\|(y)\label{sotto}   \\
&= -\int_U \varphi^2(y)(H,\BBB (y-p))d\|V\|(y) - 2\int_{U} \varphi(y)( B_\Psi(T_y\Gamma)D\varphi(y),\BBB (y-p)) d\|V\|(y)\label{sopra}.
\end{align}
We bound \eqref{sotto} from below using (USAC) for $\Psi$:
\[
C\int_U\varphi^2(y)\|T_y\Gamma - S\|^2d\|V\|(y) \le \int_{U}\varphi^2(y)\langle B_\Psi(T_y\Gamma),\BBB\rangle d\|V\|(y).
\]
Now we estimate from above \eqref{sopra}. The first addendum is estimated simply by Young inequality:
\[
-\int_U \varphi^2(y)(H,\BBB (y-p))d\|V\|(y) \le \frac{1}{2}\int_{B_2}\|H\|^2d\|V\|(y) +\frac{1}{2}\int_U\varphi^4(y)\|\BBB (y-p)\|^2d\|V\|(y).
\]
To estimate the second addendum of \eqref{sopra}, we need to use the algebraic identity \eqref{zeroprod} to rewrite
\[
\varphi(y)( B_\Psi(T_y\Gamma)D\varphi(y),\BBB (y-p)) = \varphi(y)( (B_\Psi(T_y\Gamma)- B_{\Psi}(S))D\varphi(y),\BBB (y-p))
\]
and hence, since $\Psi \in C^2(\G)$ by assumption, to bound
\begin{equation*}
\begin{split}
|\varphi(y)( B_\Psi(T_y\Gamma)D\varphi(y),\BBB (y-p))| &\le \varphi(y)\|B_\Psi(T_y\Gamma)- B_{\Psi}(S)\|\|D\varphi(y)\|\|\BBB (y-p)\|\\
&\leq C_2\varphi(y)\|T_y\Gamma- S\|\|D\varphi(y)\|\|\BBB (y-p)\|\\
&\leq \frac C4 \varphi^2(y)\|T_y\Gamma - S\|^2 +C_2 \|D\varphi(y)\|^2\|\BBB (y-p)\|^2,
\end{split}
\end{equation*}
where in the third inequality we used again Young inequality.
Combining the previous inequalities, equality \eqref{sotto}-\eqref{sopra} reads
\begin{align*}
\frac C2\int_U\varphi^2(y)\|T_y\Gamma - S\|^2d\|V\|(y) &\leq  \frac{1}{2}\int_{B_2}\|H\|^2d\|V\|(y) +\frac{1}{2}\int_U\varphi^4(y)\|\BBB (y-p)\|^2d\|V\|(y)\\
&\quad +C_2 \int_{U} \|D\varphi(y)\|^2\|\BBB (y-p)\|^2 d\|V\|(y).
\end{align*}
We conclude \eqref{Caceq} observing that
\begin{equation*}
\begin{split}
\|\BBB (y-p)\|&\overset{\eqref{BF*}}{=}\|\Psi(S)S^\perp y - S d\Psi(S)S^\perp (y-p)\|\\
&\leq  \|\Psi(S)S^\perp - S d\Psi(S) S^\perp\| \, \|S^\perp (y-p)\|\leq \|\Psi\|_{C^2} d((y-p),\pi),
\end{split}
\end{equation*}
where in the last inequality we have used the following elementary identity:
\begin{equation}\label{ident.dist}
\|S^\perp (y-p)\|=d((y-p),\pi).
\end{equation}
\end{proof}

For any $f \in L^1_{\loc}(\R^m)$, we define
\[
(f)_{z,R} := \fint_{B_R(z)}f(y)dy, \qquad (f)_R :=  \fint_{B_R(0)}f(y)dy.
\]

\begin{Cor}[Caccioppoli inequality for $u$]\label{graphCac}
Let $\Psi,u$ be as in Theorem \ref{regteo}. Then, for some constant $C_2(n,m,\|\Psi\|_{C^2},C,\|u\|_{\Lip}) > 0$, where $C$ is the constant in Definition \ref{DEFAC2} for $\Psi$, and $k = 2(1 +\|u\|_{\Lip})$, we have
\begin{equation}\label{Caceqgraph}
\fint_{B_{r}(x_0)}\|Du(x) - A\|^2dx\le \frac{C_2}{r^2}\fint_{B_{kr}(x_0)}\|u(x)-(u)_{x_0,kr} - A(x-x_0)\|^2dx + r^2C_2\fint_{B_{kr}(x_0)}\|H'\|^2dx,
\end{equation}
for all $A \in \R^{n\times m}$ with $\|A\| \le 2\|Du\|_\infty$, $x_0 \in \Omega$, and $r > 0$ such that $\dist(x_0, \partial\Omega) < \frac{k}{2}r$. Here,
\[
H'(x) := H(x,u(x)) \in L^2(\Omega).
\]
\end{Cor}
\begin{proof}
In this proof we will denote with $C_2$ a positive constant which may change line by line, but that shall always depend just on $n,m,\|\Psi\|_{C^2},C,\|u\|_{\Lip}$.
By \eqref{Caceq}, we know:
\[
\fint_{B_{R}(z)}\|T_y\Gamma_u - S\|^2d\|V\|(y) \le \frac{C_2}{R^2}\fint_{B_{2R}(z)}\dist(y-p,\pi)^2d\|V\|(y) + R^2C_2\fint_{B_{2R}(z)}\|H\|^2d\|V\|(y),
\]
for every $p\in \R^N$ and $S \in \G, z \in U, R >0$ such that $\dist(z,\partial U) \le 4R$, where $\pi=\text{Im}(S)$. Here, $V = \llbracket \Gamma_u\rrbracket$. We fix $A \in \R^{n\times m}$, and consequently choose $S = h(A)$, where $h$ is the map defined in \eqref{h}. We choose $p=(x_0,(u)_{x_0,kr})$ and we consider $r> 0$ as in the statement of the theorem fixed. Define also
\[
\Gamma_u^r = \{(x,u(x)): x \in B_r(x_0)\}.
\]
Notice that we make a small abuse of notation, denoting with the same symbol balls in $\R^N$ and $\R^m$. Let $L := \|u\|_{\Lip}$. We notice the following:
\begin{equation}\label{cont}
\Gamma_{u}^{\frac{r}{1 + L}} \subset B_r((x_0,u(x_0)))\cap \Gamma_u \subset B_{2r}((x_0,u(x_0)))\cap \Gamma_u \subset \Gamma_{u}^{2r}.
\end{equation}
From the area formula, we also see that
\begin{equation}\label{conto1}
 \mathcal{H}^m(B_r((x_0,u(x_0)))\cap \Gamma_u) \leq C_2 r^m ,\quad C_2 r^m \leq \mathcal{H}^m(B_{2r}((x_0,u(x_0)))\cap \Gamma_u )
\end{equation}
%
Hence, we rewrite \eqref{Cac} for $R = r$ and $z = (x_0,u(x_0))$ in the following form:
\[
\fint_{B_{r}(z)\cap \Gamma_u}\|T_y\Gamma_u - S\|^2d\mathcal{H}^m(y) \le \frac{C_2}{r^2}\fint_{B_{2r}(z)\cap \Gamma_u}\dist(y-p,\pi)^2d\mathcal{H}^m + r^2C_2\fint_{B_{2r}(z)\cap \Gamma_u}\|H\|^2d\mathcal{H}^m
\]
and then use \eqref{cont},\eqref{conto1} to write
\begin{equation}\label{midCac}
\fint_{\Gamma_{u}^{\frac{r}{1 + L}}}\|T_y\Gamma_u - S\|^2d\mathcal{H}^m(y) \le \frac{C_2}{r^2}\fint_{\Gamma_{u}^{2r}}\dist(y-p,\pi)^2d\mathcal{H}^m + r^2C_2\fint_{\Gamma_{u}^{2r}}\|H\|^2d\mathcal{H}^m.
\end{equation}
Now we use the area formula to rewrite and estimate \eqref{midCac}, to finally obtain \eqref{Caceqgraph}. Firstly, the area formula yields
\[
\fint_{\Gamma_{u}^{\frac{r}{1 + L}}}\|T_y\Gamma_u - S\|^2d\mathcal{H}^m(y) = \fint_{B_{\frac{r}{1+L}}(x_0)}\|h(Du(x))-h(A)\|^2\mathcal{A}(Du(x))dx.
\]
Now, $h: \R^{n\times m} \to \G$ is invertible on the set $E\subset \G$ defined as
\[
E = \{T \in \G: \det(T') \neq 0\},
\]
where $T'$ is the $m\times m$ submatrix obtained by $T$ only considering the first $m$ rows and columns. Moreover, $h^{-1}: E \to \R^{n\times m}$ is locally Lipschitz, hence 
\[
C_2\|h(Du(x))-h(A)\|\ge \|Du(x)-A\|.
\]
Since $\mathcal{A}(X) \ge 1$ for every  $X \in \R^{n\times m}$, we can finally bound
\begin{equation}\label{below}
\fint_{B_{\frac{r}{1+L}}(x_0)}\|Du(x)-A\|^2dx \le C_2\fint_{B_{\frac{r}{1+L}}(x_0)}\|h(Du(x))-h(A)\|^2\mathcal{A}(Du(x))dx.
\end{equation}
Now we wish to estimate from above the addendum
\[
\fint_{B_{2r}(z)\cap \Gamma_u}\dist(y-p,\pi)^2d\mathcal{H}^m.
\]
First of all,
\[
\fint_{B_{2r}(z)\cap \Gamma_u}\dist(y-p,\pi)^2d\mathcal{H}^m \overset{\eqref{cont}-\eqref{conto1}}{\le} C_2\fint_{\Gamma_u^{2r}}\dist(y-p,\pi)^2d\mathcal{H}^m.
\]
Secondly, as in \eqref{ident.dist}, we write
\begin{equation}\label{fir}
\dist(y-p,\pi) = \|S^\perp(y-p)\| = \left\|(\id -h(A))\left(\begin{array}{c}x-x_0\\ u(x)-(u)_{x_0,kr}\end{array}\right)\right\|.
\end{equation}
Now we claim that
\begin{equation}\label{sec}
\left\|(\id -h(A))\left(\begin{array}{c}x-x_0\\ u(x)-(u)_{x_0,kr}\end{array}\right)\right\| = \left\|(\id -h(A))\left(\begin{array}{c}0\\ u(x)-(u)_{x_0,kr} - A(x-x_0)\end{array}\right)\right\|.
\end{equation}
Indeed
\[
(\id-h(A))\left(\begin{array}{c}\id_m\\ A\end{array}\right)\overset{\eqref{h}}{=}(\id-h(A))M(A) \overset{\eqref{h}}{=} (\id_m-M(A)(M(A)^tM(A))^{-1}M(A)^t)M(A) = 0.
\]
Combining \eqref{fir} and \eqref{sec}, we estimate
\begin{equation}\label{tir}
\dist(y-p,\pi) \le C_2\|u(x)-(u)_{x_0,kr}-A(x-x_0)\|.
\end{equation}
In particular, this allows us to write
\begin{equation}\label{secbound}
\begin{split}
\fint_{B_{2r}(z)\cap \Gamma_u}\dist(y-p,\pi)^2d\mathcal{H}^m &\overset{\eqref{cont}-\eqref{conto1}}{\le} C_2\fint_{\Gamma_u^{2r}}\dist(y-p,\pi)^2d\mathcal{H}^m\\
&\quad \;\, = C_2\fint_{B_{2r}(x_0)}\dist( (x,u(x))-(x_0,(u)_{x_0,kr}),\pi)\mathcal{A}(Du(x))dx\\
&\quad \overset{\eqref{tir}}{\le} C_2\fint_{B_{2r}(x_0)}\|u(x)-(u)_{x_0,kr}-A(x-x_0)\|dx,
\end{split}
\end{equation}
where in the first equality we used the area formula, and to get the second inequality we used, other than \eqref{tir}, also the fact that $u \in \Lip$ to bound $\mathcal{A}(Du(\cdot))$ with a constant depending on $L$. Finally
\begin{equation}\label{thirbound}
\fint_{\Gamma^{2r}_u}\|H\|^2dx =\fint_{B_{2r}(x_0)}\|H'\|^2\mathcal{A}(Du(x))dx \le C_2\fint_{B_{2r}(x_0)}\|H'\|^2dx,
\end{equation}
that once again exploits the area formula and the fact that $u \in \Lip$. Inequalities \eqref{midCac}-\eqref{below}-\eqref{secbound}-\eqref{thirbound} imply \eqref{Caceqgraph}.
\end{proof}

We will now prove a decay of the following classical quadratic excess: $$E(x,r) := \fint_{B_r(x)}\|Du(x) - (Du)_{x,r}\|^2 dx.$$

\begin{prop}[Excess decay]\label{Tilt}
Let $\Psi,u,H$ be as in Theorem \ref{regteo}. Let moreover $k = 2(1 + \|u\|_{\Lip}) > 0$ as in Corollary \ref{graphCac}. Then, there exists a constant $c > 0$ such that for every $\tau \in \left(0,\frac{1}{4k}\right)$, there exists $\varepsilon = \varepsilon(\tau)>0$ such that
\[
E(x,r) \le \varepsilon(\tau) \quad \text{ and } \quad r^{1-\frac mp}\|H'\|_p \le E(x,r)
\]
imply
\begin{equation}\label{decay}
E(x,\tau r) \le c\tau^2E(x,r)
\end{equation}
for every $B_r(x) \subset \Omega$.
\end{prop}

\begin{proof}
The proof is analogous to \cite[Lemma 4.1]{EVA}. We adapt it below to our setting. The key point is a contradiction blow-up argument, that uses the regularity theory for the linearized problem and the Caccioppoli estimate \eqref{Caceqgraph}.
\\
\\
Suppose the thesis were false. Then, for every $c>0$ there exists $\tau \in \left(0,\frac{1}{4k}\right)$ and a sequence of points $\{x_j\} \subset \Omega$ and radii $r_j > 0$ such that
\begin{equation}\label{contradict}
E(x_j,r_j) = \lambda_j^2 \to 0 \quad \text{ and }\quad r^{1-\frac mp}_j\|H'\|_p \le E(x_j,r_j) = \lambda_j^2
\end{equation}
but
\begin{equation}\label{contradict1}
E(x_j,\tau r_j) \ge c\tau^2E(x_j,r_j)
\end{equation}
We consider blow-ups of $u$ of the following form
\[
v_j(z) := \frac{u(x_j + r_jz) - (u)_{x_j,r_j} - r_jA_jz}{\lambda_jr_j}, \qquad \mbox{with $ A_j:= (Du)_{x_j,r_j}$.}
\]
For every $j \in \mathbb{N}$ the maps $v_j : B_2(0)\subset \R^m \to \R^n$ enjoy the following properties:
\begin{enumerate}[(i)]
\item $D v_j(z) = \frac{D u(x_j + r_jz) - A_j}{\lambda_j}$;
\item $(v_j)_{1} = 0$, $(D v_j)_{1} = 0$;
\item $\displaystyle\fint_{B_1(0)}\|D v_j\|^2 = 1 $;\label{L2bound}
\item $\displaystyle\fint_{B_1(0)}\|v_j\|^2 \le \gamma $, for some $\gamma > 0$;\label{poinc}
\item $\displaystyle\fint_{B_\tau(0)}\|D v_j - E_j\|^2 \ge c\tau^2\fint_{B_1(0)}\|D v_j\|^2 = c\tau^2$, \quad where $E_j := (D v_j)_{\tau}$. \label{cine}
\end{enumerate}
The first three conditions are easy consequences of the definition of $v_j$, the fourth is an application of Poincar\'e's inequality, and the fifth can be seen from \eqref{contradict1} and the definition of $v_j$. \eqref{L2bound}-\eqref{poinc} imply that, up to a non-relabeled subsequence, we can assume
\[
v_j \rightharpoonup v \text{ in }W^{1,2}(B_1,\R^n),\qquad  v_j \to v \text{ in }L^2(B_1,\R^n) 
\]
and, since $\{A_j\}_j$ is equibounded, $A_j \to A \in \R^{n\times m}$. Recalling the definition of $F_\Psi$ as in \eqref{functional}, we define the sequence of integrands
\begin{equation}\label{FJ}
F_j(X) := \frac{1}{\lambda_j^2}(F_\Psi(\lambda_jX + A_j) - F_\Psi(A_j) - \lambda_j\langle DF_\Psi(A_j),X\rangle),
\end{equation}
then, $F_j(X) \to D^2F_\Psi(A)[X,X]$ locally in the $C^2$ topology. We claim that $v$ is a critical point for the functional
\[
G(X) := D^2F_\Psi(A)[X,X],
\]
or in other words
\begin{equation}\label{D2}
\int_{B_1}D^2F_\Psi(A)[Dv,Dg]dx = 0,\quad\forall g \in C^\infty_c(B_1,\R^n).
\end{equation}
To see this, fix a test vector-field $g$ and use \cite[Proposition 6.8]{DLDPKT} to find that, since $\llbracket\Gamma_u\rrbracket$ has $\Psi$-mean curvature $H$ bounded in $L^p$, then we find a constant $C_1 = C_1(\|H\|_{L^p}) > 0$ such that\footnote{In fact, as written at the end of the proof of \cite[Proposition 6.8]{DLDPKT}, $C_1$ can be taken to be exactly $\|H'\|_{L^p}$.} $C_1(0) = 0$, and for which
\begin{equation}\label{outervar}
\left|\int_{\Omega}\langle DF_\Psi(Du(x)), D\eta(x)\rangle dx\right| \le C_1\|\eta \A^{\frac 1{p'}}(Du)\|_{p'}, \qquad \forall \eta \in C^\infty_c(\Omega,\R^n).
\end{equation}
Here we denoted $\frac 1{p'}+\frac 1p=1$.
Now we plug in the previous inequality $\eta_j(x) = g\left(\frac{x-x_j}{r_j}\right)$. Let us rewrite the right hand side and the left hand side of \eqref{outervar} separately. We have
\begin{align*}
\int_{\Omega}\langle DF_\Psi(Du(x)), D\eta_j(x)\rangle dx &= \frac{1}{r_j}\int_{\Omega}\left\langle DF_\Psi(Du(x)), Dg\left(\frac{x - x_j}{r_j}\right)\right\rangle dx\\
&= r_j^{m - 1}\int_{B_1}\langle DF_\Psi(Du(x_j + r_jy)), Dg(y)\rangle dy\\
&=r_j^{m - 1}\int_{B_1}\langle DF_\Psi(Du(x_j + r_jy)) - DF_\Psi(A_j), Dg(y)\rangle dy\\
&= r_j^{m - 1}\int_{B_1}\langle DF_\Psi(A_j + \lambda_j Dv_j(y)) - DF_\Psi(A_j), Dg(y)\rangle dy\\
& = \lambda_jr_j^{m - 1}\int_{B_1} \langle DF_j(Dv_j(y)),Dg(y)\rangle dy,
\end{align*}
where we used the compactness of the support of $g$ for passing from the second to the third equality and the definitions of $v_j$ and $F_j$ in the rest of the equalities. Now we can turn to the right hand side of \eqref{outervar}. From now on we will denote with $C_2$ a positive constant which may change line by line, but that shall always depend just on $n,m,p,\|\Psi\|_{C^2},C,\|u\|_{\Lip},k$:
\begin{align*}
\|\eta_j \A^{\frac 1{p'}}(Du)\|_{p'} \le C_2\|\eta_j\|_{p'} = C_2 \left (\int_{\Omega}\|\eta_j(x)\|^{p'}dx \right)^{\frac 1{p'}} = C_2r^{\frac m{p'}}_j \left (\int_{B_1}\|g(y)\|^{p'}dy\right)^{\frac 1{p'}},
\end{align*}
where we bounded $\A^{\frac 1{p'}}(Du) \le {C_2}$ using the fact that $u$ is Lipschitz.
These computations allows us to rewrite \eqref{outervar} as
\[
\lambda_jr_j^{m - 1}\int_{B_1} \langle DF_j(Dv_j(y)),Dg(y)\rangle dy \le C_1 C_2r^{\frac m{p'}}_j \left (\int_{B_1}\|g(y)\|^{p'}dy\right)^{\frac 1{p'}}.
\]
Dividing by $\lambda_j r_j^{m-1}$, we see that the right hand side becomes
\[
C_1 C_2\frac{r^{1-\frac m{p}}_j}{\lambda_j} \left (\int_{B_1}\|g(y)\|^{p'}dy\right)^{\frac 1{p'}} \overset{\eqref{contradict}}{\le} \lambda_j\frac{C_1 C_2}{\|H'\|_p}\left (\int_{B_1}\|g(y)\|^{p'}dy\right)^{\frac 1{p'}} ,
\]
that converges to $0$ since $\lambda_j^2 \to 0$. Of course, if $\|H'\|_{L^p}=0$, the previous computation cannot be performed, but in that case we see that the right-hand side of \eqref{outervar} is identically $0$, since the constant appearing in \eqref{outervar} satisfies $C_1(0) = 0$. To finish the proof, we need to show that
\[
\int_{B_1}\langle DF_j(Dv_j),Dg\rangle dx \to \int_{B_1}D^2F_\Psi(A)[Dv,Dg]dx = 0.
\]
To do so,  using the definition of $F_j$ we rewrite:
\[
DF_j(Dv_j) = \frac{DF_\Psi(A_j + \lambda_jDv_j) - DF_\Psi(A_j)}{\lambda_j} = \int_0^1D^2F_\Psi(A_j + s\lambda_jDv_j(x))Dv_j(x)ds.
\]
By assumption $Dv_j$ converges weakly to $Dv$, hence it suffices to prove  that
\[
d_j(x):= \int_0^1D^2F_\Psi(A_j + s\lambda_jDv_j(x))ds \to D^2F_\Psi(A)
\]
strongly in $L^2$ to conclude. First, by the definition of $v_j$, we infer that $\{\lambda_jv_j\}_j$ is an equilipschitz sequence, and hence that $\{A_j + s\lambda_jDv_j\}_{j}$ is a sequence equibounded in $L^\infty$. Furthermore, as $\{v_j\}_j$ is equibounded in $W^{1,2}$, $\{\lambda_jv_j\}_j$ converges to $0$ strongly in $L^2$, and we may assume (up to non-relabeled subsequences) that $\lambda_jDv_j \to 0$ pointwise a.e.. Then dominated convergence implies the convergence of $d_j$ in $L^p$ for every $p \in [1,+\infty)$, and hence we find that our claim \eqref{D2} holds.
\\
\\
Since $v$ is a weak solution of a linear elliptic systems with constant coefficients, see $\eqref{D2}$, then classical elliptic regularity theory, \cite{MORB}, yields: 
\begin{equation}\label{quant}
\sup_{B_{\frac{1}{2}}(0)}\|D^2v\|^2 \le C_2\fint_{B_1}\|D v\|^2 dx \overset{\eqref{L2bound}}{\leq} C_2.
\end{equation}
Furthermore, from \eqref{Caceqgraph} and H\"older inequality we get:
\begin{equation}\label{interm}
\begin{split}
\fint_{B_{r}(x_0)}&\|Du(x) - A\|^2dx\le \frac{C_2}{r^2}\fint_{B_{kr}(x_0)}\|u(x)-(u)_{x_0,kr} - A(x-x_0)\|^2dx + r^2C_2\fint_{B_{kr}(x_0)}\|H'\|^2dx\\
&\le \frac{C_2}{r^2}\fint_{B_{kr}(x_0)}\|u(x)-(u)_{x_0,kr} - A(x-x_0)\|^2dx + C_2\left [r^{1-\frac mp} \left (\int_{B_{kr}(x_0)}\|H'\|^pdx \right)^{\frac 1p}\right]^2.
\end{split}
\end{equation}
We rewrite \eqref{interm} in terms of $v_j$ choosing $r=\tau r_j,A = (Du)_{x_j,r_jk\tau}$, simply by rescaling, translating, dividing by $\lambda_j^2$ and estimating $H'$ with its $L^p$-norm:
\begin{equation}\label{Cacpoint}
\fint_{B_{\tau}}\|Dv_j - B_j\|^2dx\le \frac{C_2}{\tau^2}\fint_{B_{k\tau}}\|v_j -b_j - B_jx\|^2dx + \frac{C_2}{\lambda_j^2}\left (r_j^{1-\frac mp} \|H'\|_p\right)^2,
\end{equation}
where $b_j:=(\lambda_jr_j)^{-1}((u)_{x_j,k\tau r_j}-(u)_{x_j,r_j})=(v_j)_{k\tau}$ and $B_j = (Dv_j)_{k\tau}$.
By \eqref{contradict}, the last addendum converges to $0$ as $j \to \infty$. It is well-known that
\begin{equation}\label{average}
\fint_{B_{\tau}}\|Dv_j - E_j\|^2dx \le \fint_{B_{\tau}}\|Dv_j - B_j\|^2dx,
\end{equation}
and hence, denoting $B= (Dv)_{k\tau}$, we obtain:
\begin{align*}
c\tau^2 &\overset{\eqref{cine}}{\le} \limsup_j\fint_{B_\tau}\|Dv_j - E_j\|dx \overset{\eqref{average}}{\le}\limsup_j \fint_{B_{\tau}}\|Dv_j - B_j\|^2dx\\
&\overset{\eqref{Cacpoint}}{\le} \frac{C_2}{\tau^2}\limsup_j\left[\fint_{B_{k\tau}}\|v_j -b_j - B_jx\|^2dx + \frac{1}{\lambda_j^2}\left (r_j^{1-\frac mp} \|H'\|_p\right)^2 \right]\\
&\overset{\eqref{contradict}}{=} \frac{C_2}{\tau^2}\fint_{B_{k\tau}}\|v -b - B x\|^2dx \le C_2\fint_{B_{k\tau}}\|Dv-B\|^2dx
\overset{\eqref{quant}}{\le}  C_2 \tau^2,
\end{align*}
where the last line is obtained by Poincar\'e inequality, using that $(v_j)_{k\tau}=b_j$ implies $(v)_{k\tau}=b$.
Choosing $c> C_2$, we obtain the desired contradiction.
\end{proof}

Now we can finally give the proof of Theorem \ref{regteo}:

\begin{proof}[Proof of Theorem \ref{regteo}:]
We define
\[
\Omega_0 := \{x \in \Omega: \lim_{r\to 0}E(x,r) = 0\}.
\]
This set is of full measure in $\Omega$ as it contains all Lebesgue points of $Du$. We want to show that $\Omega_0$ is open and $Du_{|\Omega_0} \in C^{\alpha}$ for some $\alpha \in (0,1)$. In the ongoing proof, we fix $\tau \in (0, (4k)^{-1})$, where $k = 2(1+\|u\|_{\Lip})$, satisfying $c\tau^{1-\beta} < 1$ and $\tau^\beta < \frac{1}{8}$, where $c$ is the constant found in Proposition \ref{Tilt} and $\beta := 1-\frac{m}{p}$. Further, let $x_0 \in \Omega_0$, and define the auxiliary excess:
\[
F(s) := E(x_0,s) + \Lambda s^{\beta}\|H'\|_p, \quad \Lambda := \frac{1}{8\tau^k}.
\]
We choose $r > 0$ such that
\begin{equation}\label{Eeps}
E(x_0,r) \le F(r) < \eps(\tau),
\end{equation}
where $\eps$ is given by Proposition \ref{Tilt}. From now on $r$ and $\tau$ are fixed. If $r^{\beta}\|H'\|_p \le E(x_0,r)$, by Proposition \ref{Tilt} we also find
\[
F(\tau r) = E(x_0,\tau r) + \Lambda \tau^\beta r^{\beta}\|H'\|_p \overset{\eqref{decay}}{\le} c\tau^2E(x_0,r) + \Lambda \tau^\beta r^{\beta}\|H'\|_p \overset{c\tau^\beta < 1}{\le} \tau^\beta F(r).
\]
On the other hand, if $r^{\beta}\|H'\|_p > E(x_0,r)$, we have
\begin{align*}
F(\tau r) = E(x_0,\tau r)  + \Lambda\tau^\beta r^{\beta}\|H'\|_p &\le \frac{1}{\tau^k}E(x_0,r) + \Lambda\tau^\beta r^{\beta}\|H'\|_p < (\tau^{-k}r^\beta + \Lambda\tau^\beta r^\beta)\|H'\|_p\\
&= (\tau^{-k}\Lambda^{-1} + \tau^\beta)\Lambda r^{\beta}\|H'\|_p \overset{\Lambda= 8^{-1}\tau^{-k}}{\le} (8^{-1} + \tau^\beta)\Lambda r^{\beta}\|H'\|_p\\
& \overset{\tau^\beta < 8^{-1}}{\le} \frac{1}{4}\Lambda r^{\beta}\|H'\|_p \le \frac{1}{4}F(r).
\end{align*}
In particular, if $r$ satisfies \eqref{Eeps}, we always have $F(\tau r) \le \frac 14 F(r)$. This inequality allows us to iterate the reasoning with $\tau r$  instead of $r$ (notice in fact that $\tau r$ still satisfies \eqref{Eeps}). Hence we find, for every $\ell \in \N$,
\[
F(\tau^\ell r) \le 4^{-\ell} F(r), \quad \forall \ell \in \N.
\]
From this, one easily find the existence of $\alpha \in (0,1)$ such that
\[
F(R) \le c_1R^{2\alpha}, \quad \forall R \in (0,r),
\]
for some constant $c_1 > 0$ depending only on $r$ and $\tau$. Now the key observation is that, fixed $r > 0$, for points $x$ sufficiently close to $x_0$, one still has
\[
E(x,r) + \Lambda r^\beta\|H'\|_p < \eps(\tau),
\]
that is an easy consequence of the continuity of $x \mapsto E(x,r)$. Therefore, we find that there exists $\rho > 0$ such that
\[
E(x,R) + \Lambda R^\beta\|H'\|_p < c_1 R^{2\alpha}, \quad \forall x \in B_\rho(x_0), R \in (0,r).
\]
In particular, we infer
\[
E(x,R) < c_2 R^{2\alpha}, \quad \forall x \in B_\rho(x_0), R \in (0,r).
\]
This shows that $Du : B_{\rho}(x_0) \to \R^n$ is in a Campanato space, and it is well-know that this yields H\"older regularity for $Du$, see for instance \cite[Proof of Theorem 3.2]{CAM}, and this concludes the proof.
\end{proof}

\section{Compactness}\label{s:compactness}
Aim of this section is to prove Theorem \ref{D}. This will be obtained combining the following Theorems \ref{teo:comp0} and \ref{teo:comp}. In order to precisely state Theorem \ref{D}, we can use the notion of differential inclusions. In this way, Theorem \ref{D} is equivalent to say that the only Young measures generated by \emph{div-curl} inclusions supported in
\[
K_{F_\Psi} := \left\{A \in \R^{(2n + m)\times m}:
A =
\left(
\begin{array}{c}
X\\
DF_\Psi(X)\\
X^TDF_\Psi(X) - F_\Psi(X)\id
\end{array}
\right)
\right\}
\]
are trivial in the case $\Psi$ satisfies (AC). In particular, this answers \cite[Question 9]{DLDPKT} for $\Psi$ satisfying (AC). We will not enter in the details of the theory of differential inclusions, we refer the reader to \cite[Section 2]{DLDPKT} for a thorough explanation of the terminology.

\begin{Teo}\label{teo:comp0}
Let $\Psi\in C^1(\G,(0,\infty))$ and $F_\Psi$ as in \eqref{functional}. Let $\Omega \subset \R^m$ be open and bounded. Consider sequences $u_j: \Omega \to \R^n$, $A_j: \Omega \to \R^{n\times m}$, $B_j: \Omega \to \R^{m\times m}$ such that $u_j$ is equibounded in $W^{1,\infty}$ and $A_j,B_j$ are equibounded in $L^\infty$. Suppose further that $\dv A_j$ and $\dv B_j$ are equibounded in $L^1$. Define
\[
W_j := \left(\begin{array}{cc}Du_j\\ A_j \\ B_j\end{array}\right)
\]
and suppose $\dist(W_j,K_{F_\Psi}) \to 0$ pointwise a.e. as $j \to \infty$. Suppose further that $u_j \rightharpoonup u$ in $W^{1,2}(\Omega,\R^n)$. Then, the associated varifolds $V_j := \llbracket\Gamma_{u_j}\rrbracket$ converge in the sense of varifolds (i.e. weakly$^*$ as measures on $\Omega\times\R^n\times \mathbb{G}(N,m))$ to $V = \llbracket\Gamma_u\rrbracket$.
\end{Teo}

\begin{Teo}\label{teo:comp}
Let $\Omega \subset \R^m$ be open and bounded. Let $u_j: \Omega \subseteq\R^m \to \R^n$ be a sequence of maps such that $u_j \rightharpoonup u$ in $W^{1,p}(\Omega,\R^n)$ and $D u_j \weak (\nu_x)_x$ in the sense of Young measures. Suppose the graphs $\llbracket\Gamma_{u_j}\rrbracket$ converge in the sense of varifolds to $\llbracket\Gamma_u\rrbracket$. Then $u_j$ converges to $u$ in the strong $W^{1,q}(\Omega,\R^n)$-topology for every $1\le q < p$. 

Conversely, if $u_j: \Omega \subseteq\R^m \to \R^n$ satisfies $u_j \to u$ strongly in $W^{1,m}(\Omega,\R^n)$, then there exists a subsequence $u_{j_k}$ such that $\llbracket \Gamma_{u_{j_k}}\rrbracket$ converges in the sense of varifolds to $\llbracket \Gamma_{u}\rrbracket$.
\end{Teo}

\begin{proof}[Proof of Theorem \ref{teo:comp0}]

Since $u_j$ is equibounded in $W^{1,\infty}$, then $u_j \to u$ uniformly by Ascoli-Arzel\`a compactness criterion and $V_j$ is an equibounded sequence of measures. Therefore, up to extracting a subsequence, $V_j \weak V$. We need to prove that $V = \llbracket\Gamma_u\rrbracket$. First, we claim that 
\begin{equation}\label{claim1}
\spt(\|V\|) \subseteq \Gamma_u.
\end{equation} 
To see \eqref{claim1}, fix $y \in \R^N \setminus \Gamma_u$ and $r>0$ such that $B_r(y) \cap \Gamma_u = \emptyset$. Then, by the uniform convergence of $u_j$ to $u$, up to taking $j$ large enough
\[
\|V_j\|(B_r(y)) = 0.
\]
Since $\|V_j\|\weak \|V\|$, by lower semicontinuity we deduce that $\|V\|(B_r(y)) = 0$. In particular $y \in \R^N \setminus \spt(\|V\|)$ and we deduce \eqref{claim1}. 

We want to prove that the limit varifold $V$ is integer rectifiable. We wish to apply \cite[Theorem 4.1]{DeR2016}. 

To this aim, we show that $[\delta_\Psi V]$ is a Radon measure. This is an easy consequence of \cite[Lemma 7.3]{DLDPKT}, that yields for every $g = (g^1,\dots, g^{N})\in C^1_c(\Omega\times\R^n)$:
\begin{equation}\label{structure}
[\delta_{\Psi}(\llbracket\Gamma_{u_j}\rrbracket)](g) = \int_{\Omega}\langle B(Du_j(x)), D(g_1(x,u_j(x)))\rangle dx +\int_{\Omega}\langle A(Du_j(x)), D(g_2(x,u_j(x)))\rangle dx,
\end{equation}
where $g_1(x,y) := (g^1(x,y),\dots, g^m(x,y))$, $g_2(x,y) := (g^{m + 1}(x,y),\dots, g^{m + n}(x,y))$ and $A(X)$ and $B(X)$ are defined by $A(X) = DF_\Psi(X), B(X) = X^TDF_\Psi(X) - F_\Psi(X)\id$. Since $W \mapsto [\delta_{\Psi}(W)]$ is continuous with respect to the weak-$*$ convergence of varifolds, the left hand side of \eqref{structure} converges to $[\delta_{\Psi}(V)](g)$ as $j\to \infty$, while the right hand side can be rewritten as:
\begin{equation}\label{structure1}
\begin{split}
\int_{\Omega}\langle B(Du_j(x)) - B_j(x), D(g_1(x,u_j(x)))\rangle dx +\int_{\Omega}\langle A(Du_j(x))-A_j(x), D(g_2(x,u_j(x)))\rangle dx\\
+\int_{\Omega}\langle B_j(x), D(g_1(x,u_j(x)))\rangle dx +\int_{\Omega}\langle A_j(x), D(g_2(x,u_j(x)))\rangle dx.
\end{split}
\end{equation}
We may assume, up to passing to a non-relabeled subsequence, that $\dv B_j$ and $\dv A_j$ weakly-$*$ converge as Radon measures to $\mu$ and $\nu$, respectively. Therefore, taking the limit as $j\to \infty$, since $\dist(W_j,K_{F_\Psi}) \to 0$ pointwise a.e. and $g_1(x,u_j(x))-g_1(x,u(x))$, $g_2(x,u_j(x))-g_2(x,u(x))$ converge uniformly to $0$, then \eqref{structure1} converges to
\[
-\int_{\Omega}g_1(x,u(x))d\mu(x) -\int_{\Omega}g_2(x,u_j(x))d\nu(x).
\]
Hence, $[\delta_\Psi V]$ is a Radon measure.
\\
\\
Secondly, in order to apply \cite[Theorem 4.1]{DeR2016}, we need to prove that 
\begin{equation}\label{dens}
\theta^m_*(y, \|V\|)>0, \qquad \mbox{for $\|V\|$-a.e. $y \in \Omega\times\R^n$}.
\end{equation}
Let $\pi: \Omega\times \R^n \to \Omega$ be the projection map on the first factor. 
Now, fix a point $z \in \Omega$. For every $j \in \mathbb N$ and for every $ r \in (0,\dist(z,\partial\Omega))$ it holds
\begin{equation}\label{ineqj}
 \pi_\#\|V_j\|(B_r(z)) =\|V_j\|(\pi^{-1}(B_r(z)))= \int_{B_r(z)} \mathcal A(D u_j)(x)  dx \ge |B_r(z)|= r^m\omega_m.
\end{equation}
As $ \pi_\#\|V_j\|\weak \pi_\#\|V\|$, for every $r \in  (0,\dist(z,\partial\Omega))$ we compute
\begin{equation}\label{final1}
\frac{r^m}{2^m}\omega_m \overset{\eqref{ineqj}}{\leq} \limsup_j \pi_\#\|V_j\|(\overline{B_\frac{r}{2}(z)}) \le \pi_\#\|V\|(\overline{B_\frac{r}{2}(z)})\le \pi_\#\|V\|({B_r(z)}),
\end{equation}
where the second inequality can be found in \cite[Theorem 1.40]{EVG}. Denote $L := \|u\|_{\Lip}$, $ z = \pi(y)\in \Omega$ and fix any $0< r<\dist(z,\partial\Omega)$.  We can compute:
\[
\Gamma_u\cap \pi^{-1}\left(B_{\frac{r}{\sqrt{1 + L^2}}}(z)\right) = \Gamma_u \cap \left(B_{\frac{r}{\sqrt{1 + L^2}}}(z)\times \R^n\right) \subseteq \Gamma_u\cap{B_{r}(y)} 
\] 
Hence:
\begin{align*}
 \pi_\#\|V\|\left(B_{\frac{r}{\sqrt{1 + L^2}}}(z)\right) &= \|V\|\left(\pi^{-1}\left (B_{\frac{r}{\sqrt{1 + L^2}}}(z)\right)\right) = \|V\|\left (\Gamma_u\cap\pi^{-1}\left(B_{\frac{r}{\sqrt{1 + L^2}}}(z)\right)\right)\\
 &\le \|V\|(\Gamma_u\cap B_{r}(y)) = \|V\|(B_{r}(y)).
\end{align*}
This estimate, combined with $\eqref{final1}$, implies that there exists $c = c(m) > 0$ such that
\begin{equation}\label{final2}
cr^m   \le \pi_\#\|V\|\left(B_{\frac{r}{\sqrt{1 + L^2}}}(z)\right) \le  \|V\|(B_{r}(y)), \quad \forall 0< r<\dist(z,\partial\Omega).
\end{equation}
As remarked above, this is enough to conclude \eqref{dens}. Moreover \eqref{final2} implies that $\Gamma_u \subseteq \spt(\|V\|)$, 
which combined with \eqref{claim1} gives
\begin{equation}\label{final4}
\Gamma_u = \spt(\|V\|).
\end{equation}
We apply \cite[Theorem 4.1]{DeR2016} to deduce that $V$ is an integer rectifiable varifold. Moreover, $\eqref{final4}$ tells us that $V = (\Gamma_u,\theta)$. 
Exploiting the graphicality of the sequence of varifolds, we have the equality
\[
\llbracket\Omega\times \{0\}^n\rrbracket = \pi^{\#}V_j \weak \pi^{\#}V=(\Omega\times\{0\}^n,\theta \circ (\pi_{|\Gamma_u})^{-1}).
\]
This implies that
\begin{equation}\label{pfalto}
\pi^\#V = \llbracket\Omega\times \{0\}^n\rrbracket,
\end{equation}
and in particular that $\theta(y) = 1$ for $\mathcal{H}^m$-a.e. $y \in \spt(\|V\|)$.
\end{proof}

\begin{proof}[Proof of Theorem \ref{teo:comp}]
In order to prove that $u_j$ converges to $u$ in the strong $W^{1,q}(\Omega,\R^n)$-topology for every $1\le q < p$, by \eqref{comp}, it is sufficient to prove that $\nu_x = \delta_{D u(x)}$ for $\mathcal{H}^m$-a.e. $x \in \Omega$. To do so, we consider the function $h$ defined in \eqref{h}.
Note that
\[
\|h(\Lambda)\|_{\op} \le 1,\qquad \forall \Lambda \in \R^{n\times m},
\]
where $\|\cdot\|_{\op}$ is the operator norm of a linear application. This implies that
\[
\|h(\Lambda)\A(\Lambda)\|_{\op} \le  \A(\Lambda) \le C(1 + \|\Lambda\|^m), \qquad \forall \Lambda \in \R^{n\times m}.
\]
By the definition of Young measures, we have, for any $\eta\in C^{\infty}_c(\Omega)$,
\begin{equation}\label{YOUNG}
\int_{\Omega}h(D u_j)(x)\A(D u_j)(x)\eta(x) dx \to \int_{\Omega}\left(\int_{\R^{n\times m}}h(\Lambda)\A(\Lambda)d\nu_x(\Lambda)\right)\eta(x) dx.
\end{equation}
We choose the following function $f : \Omega\times\R^n\times \mathbb{G}(n + m,m) \to \R^{N\times N}$:
\[
f((x,y),T) := \eta(x)\psi(y)T,
\]
where $\psi \in C^{\infty}_c(\R^n)$ enjoys the following properties: 
\[
\psi \equiv 1 \text{ on } B_{2M}(0), \quad \mbox{ and } \quad \psi \equiv 0 \text{ on } \R^n \setminus B_{3M + 1}(0), \qquad \mbox{where $M = \max_{\spt(\eta)}\|u\|$}.
\]
With a slight abuse of notation, we will anyway write $V_j(f)$, even though $f$ is matrix-valued. This can be easily corrected considering the composition between $f$ and the function $a_{ij}: \R^{n\times m} \to \R$ that gives the $(i,j)$ component of the matrix. Since $V_j := \llbracket\Gamma_{u_j}\rrbracket \rightharpoonup V:=\llbracket\Gamma_{u}\rrbracket$ as varifolds, we deduce
\[
 \int_{\Gamma_{u_j}}T_x\Gamma_{u_j}\eta(x)\psi(y)d\mathcal{H}^m(x,y)= V_j(f) \to V(f) =  \int_{\Gamma_{u}}T_x\Gamma_{u}\eta(x)\psi(y) d\mathcal{H}^m(x,y),
\]
which, by the area formula \cite[Theorem 1.2]{MAL}, reads
\[
\int_{\Omega}h(D u_j(x))\A(D u_j(x))\psi(u_j(x))\eta(x) dx \to \int_{\Omega}h(D u(x))\A(D u(x))\psi(u(x))\eta(x) dx.
\]
By our choice of $\psi$ and the uniform convergence $u_j\to u$, we rewrite the previous limit as
\begin{equation}\label{YOUNG2}
\int_{\Omega}h(D u_j(x))\A(D u_j(x))\eta(x) dx \to \int_{\Omega}h(D u(x))\A(D u(x))\eta(x) dx.
\end{equation}
Combining \eqref{YOUNG2} with $\eqref{YOUNG}$, we deduce that for $\mathcal{H}^m$-a.e. $x \in \Omega$
\[
h(D u(x))\A(D u(x)) = \int_{\R^{n\times m}}h(\Lambda)\A(\Lambda)d\nu_x(\Lambda).
\]
Taking the inner product with $h(D u(x))^\perp$, we obtain
\begin{equation*}
\begin{split}
0&=\langle h(D u(x))^\perp, h(D u(x))\rangle \A(D u(x)) = \int_{\R^{n\times m}}\langle h(D u(x))^\perp, h(\Lambda)\rangle\A(\Lambda)d\nu_x(\Lambda)\\
&\overset{\eqref{idTS}}{=}\frac 12 \int_{\R^{n\times m}}\| h(D u(x))-h(\Lambda)\|^2\A(\Lambda)d\nu_x(\Lambda).
\end{split}
\end{equation*}
Since, $\A(\Lambda)>0$ for every $\Lambda\in \R^{n\times m}$, we deduce that
\begin{equation}
\nu_x = \delta_{D u(x)},\quad  \text{ for $\mathcal{H}^m$-a.e. } x \in \Omega
\end{equation}
as claimed. 
\\
\\
Conversely, if we have the strong convergence in $W^{1,m}$ of $u_j$ to $u$, we can also suppose, up to a non-relabeled subsequence, that $u_j$ and $D u_j$ converge pointwise a.e. to $u$ and $D u$ respectively. Therefore, for every $f \in C_c(\R^{N}\times \G)$,
\begin{align*}
\llbracket \Gamma_{u_j}\rrbracket(f) &=\int_{\Gamma_{u_j}}f(y,T_y\Gamma_{u_j})d\mathcal{H}^m(x)\\
& = \int_{\Omega}f(x,u_j(x),h(D u_j(x)))\A(D u_j(x)) dx \to \int_{\Omega}f(x,u(x),h(D u(x)))\A(D u(x)) dx\\
& = \int_{\Gamma_{u}}f(y,T_y\Gamma_{u})d\mathcal{H}^m(x) = \llbracket \Gamma_{u}\rrbracket(f),
\end{align*}
where the limit in the second line can be taken by the strong convergence of $u_j$ to $u$.
\end{proof}

\section{Example: \texorpdfstring{$l^p$}{lp} norms}\label{s:ex}

In this final section, we provide explicit examples of integrands $\Psi : \mathbb{G}(4,2) \to (0,\infty)$ satisfying (AC1). In order to do so, we use another well-know representation of $\mathbb{G}(4,2)$ in terms of the simple 2-vectors\footnote{More generally, one can identify $\mathbb{G}(N,m)$ with the space of simple $m$-vectors of $\R^N$, see \cite[Section 2.1]{MGS1}.}. Recall that $\Lambda_{2}(\R^{4})$ is the space of $2$-vectors of $\R^{4}$, i.e. the vector space given by finite linear combinations of elements of the form $v_{1}\wedge v_2$, with $v_i \in \R^{4}$.
We also let $\Lambda_{2}^s(\R^{4})\subset \Lambda_{2}(\R^{4})$ be the space of simple $2$-vectors, i.e. all elements $\tau \in \Lambda_{2}(\R^{4})$ such that
$\tau = v_1\wedge v_2.$ Let $\{e_1,e_2,e_3,e_4\}$ be a canonical basis of $\R^{4}$, then $\left\{e_i\wedge e_j: 1\le i < j \le 4\right\}$ is a \emph{canonical} basis of $\Lambda_{2}(\R^{4})$. The vector space $\Lambda_2(\R^{4})$ can be endowed with a scalar product that is defined on simple vectors as
\[
\langle v_1\wedge v_2,w_1\wedge w_2\rangle := \det(X),
\]
where $X \in \R^{2\times 2}$ is defined as $X_{ij} = (v_i,w_j)$. Consequently, we define $\|\tau\| := \sqrt{\langle \tau,\tau\rangle}$.
\\
\\
Non-zero simple vectors are in natural surjection with $\mathbb{G}(4,2)$, as for every non-zero $\tau = v_1\wedge v_2 \in \Lambda_{2}^s(\R^{4})$, one can associate the projection on $\spn\{v_1,v_2\}$. We will denote such a projection matrix onto $\spn\{v_1,v_2\}$ as $T(v_1\wedge v_2)$. For any even and $1$-homogeneous function $\mathcal{G}:\R^6 \to \R$, we can define $\Phi: \Lambda_2(\R^4) \to \R$ as:
\begin{equation}\label{GPLU}
\Phi(\tau) = \mathcal{G}(\langle v_1\wedge v_2, e_1\wedge e_2\rangle,\langle v_1\wedge v_2, e_1\wedge e_3\rangle,\langle v_1\wedge v_2, e_1\wedge e_4\rangle,\langle v_1\wedge v_2, e_2\wedge e_3\rangle,\langle v_1\wedge v_2, e_2\wedge e_4\rangle, \langle v_1\wedge v_2, e_3\wedge e_4\rangle).
\end{equation}
for $\tau = v_1\wedge v_2$. For ease of notation, denote $v_{ij} := \langle v_1\wedge v_2, e_i\wedge e_j\rangle$, for every $1\le i,j\le 4$. Notice that $v_{ij} = -v_{ji}$. Every such function $\Phi$ provides a well defined function $\Psi_\mathcal{G}: \mathbb{G}(4,2)\to \R$, simply by setting
\[
\Psi_\mathcal{G}(T) := \Phi\left(\frac{v_1\wedge v_2}{\|v_1\wedge v_2\|}\right),\quad \mbox{if $T = T(v_1\wedge v_2)$}.
\]
In this case, we will denote $T = T(v_1\wedge v_2)$. Notice that the evenness and 1-homogeneity of $\mathcal{G}$ imply that $\Psi$ is well-defined. 
\\
\\
The main result of this section is the following:

\begin{Teo}\label{Lp}
Let $p \in (1,+\infty)$, and let $\mathcal{G} = \|\cdot\|_{\ell^p}$ in $\R^6$. Then, $\Psi_{\mathcal{G}}$ satisfies (AC1).
\end{Teo}
Notice that in the case $p = 2$, we recover the area functional. To prove the previous theorem, first we need to understand how the (AC) condition \eqref{atomica} reads with respect to these coordinates. We have:
\begin{lemma}\label{Beq}
Let $v_1,v_2 \in \R^4$ be linearly independent. Then
\[
B_{\Psi_{\mathcal{G}}}(T(v_1\wedge v_2)) = B_{\mathcal{G}}\left(\frac{v_1\wedge v_2}{\|v_1\wedge v_2\|}\right),
\]
where, for all $1\le a,b \le 4$, denoting with $\partial_{(ab)}\mathcal{G}$ the partial derivative of $\mathcal{G}$ in the component $(ab)$,
\begin{equation}\label{components}
(B_{\mathcal{G}})_{ab}(v_1\wedge v_2) := \sum_{a < j} \partial_{(aj)}\mathcal{G}(v_{12},v_{13},v_{14},v_{23},v_{24},v_{34})v_{bj} + \sum_{a > j} \partial_{(ja)}\mathcal{G}(v_{12},v_{13},v_{14},v_{23},v_{24},v_{34})v_{jb}.
\end{equation}
\end{lemma}
\begin{proof}
By \cite[Lemma A.1]{HRT}, one can compute the variation of a varifold\footnote{The computation in \cite{HRT} is actually carried out for currents, but the evenness of $\mathcal{G}$ allows us to immediately extend it to varifolds.} $V = (\Gamma,\theta)$:
\begin{equation}\label{varcoor}
\begin{split}
[\delta_\Phi V](g) &= \sum_{i = 1}^2\int_{\Gamma}\langle \nabla\Phi(w_1\wedge w_2(x)),Dg(x)w_1(x)\wedge w_2(x) + w_1(x)\wedge Dg(x)w_2(x)\rangle \theta(x)d\mathcal{H}^2(x),
\end{split}
\end{equation}
where $\spn\{w_1(x),w_2(x)\}$ is the approximate tangent space to $\Gamma$ at $\mathcal{H}^m\llcorner \Gamma$-a.e. $x$, $\|w_1(x)\wedge w_2(x)\|=1$ at $\mathcal{H}^m\llcorner \Gamma$-a.e. $x$, and 
\[
\nabla \Phi(v_1\wedge v_2) = \sum_{i < j}\partial_{(ij)}\mathcal{G}(v_{12},v_{13},v_{14},v_{23},v_{24},v_{34}) e_i\wedge e_j.
\]
We define $B_{\mathcal{G}}(\tau)$ by requiring that for every $L \in \R^{4\times 4}$:
\begin{equation}\label{linapp}
\langle \nabla\Phi(v_1\wedge v_2), (Lv_1)\wedge v_2 + v_1\wedge (Lv_2)\rangle = \langle B_{\mathcal{G}}(v_1\wedge v_2), L\rangle.
\end{equation}
Some simple but lengthy computations yield that for $1\le a,b \le 4$, $B_{\mathcal{G}}(v_1\wedge v_2)$ has components given by \eqref{components}. Through \eqref{linapp}, we write
\[
[\delta_\Phi V](g) = \int_{\Gamma}\langle B_{\mathcal{G}}(w_1\wedge w_2),Dg(x)\rangle\theta(x)dx.
\]
Since $\Psi_{\mathcal{G}}(T(v_1\wedge v_2)) = \Phi\left(\frac{v_1\wedge v_2}{\|v_1\wedge v_2\|}\right)$, it is simple to see that
\[
[\delta_\Psi V](g)=[\delta_\Phi V](g), \qquad \forall g \in C_c^{1}(\R^N,\R^N).
\]
By \eqref{varB} and the arbitrarity of $V$ and $g$, this yields $B_\Psi(T(v_1\wedge v_2)) = B_\mathcal{G}(v_1\wedge v_2)$, and hence proves the Lemma.
\end{proof}
Finite measures on $\mathbb{G}(4,2)$ and even measures on the space
\[
G := \{v_1\wedge v_2 \in \Lambda^s_2(\R^4): \|v_1\wedge v_2\| = 1\}
\]
are in bijective correspondence, through the map $v_1\wedge v_2 \mapsto T(v_1\wedge v_2)$. This consideration and the previous Lemma imply the following:
\begin{lemma}\label{char}
The map $\Psi_{\mathcal{G}}: \mathbb{G}(4,2) \to \R$ satisfies (AC1) if and only for every even probability measure $\mu$ on $G$,
\[
\dim\Ker \int_G B_\mathcal{G}(v_1\wedge v_2)d\mu(v_1\wedge v_2)  \le 2.
\]
The map $\Psi_{\mathcal{G}}: \mathbb{G}(4,2) \to \R$ satisfies (AC2) if and only if for every even probability measure $\mu$ on $G$,
\[
\dim\Ker \int_G B_\mathcal{G}(v_1\wedge v_2)d\mu(v_1\wedge v_2)  = 2\qquad  \Longleftrightarrow \qquad \mu = \frac{\delta_{\tau_0} + \delta_{-\tau_0}}{2} \quad \mbox{for some $\tau_0 = v_1^0\wedge v_2^0 \in G$}.
\]
\end{lemma}

\begin{proof}[Proof of Theorem \ref{Lp}]
Fix $p \in (1,+ \infty)$. We use the characterization given by Lemma \ref{char}: given an even probability measure $\mu$ on $G$, we denote
\[
A(\mu) = \int B_{\mathcal{G}}(v_1\wedge v_2) d\mu(v_1\wedge v_2).
\]
We need to show that $\rank (A(\mu)) \ge 2$. First of all, by formula \eqref{components}, we have
\[
(B_{\mathcal{G}})_{ab}(v_1\wedge v_2) = \frac{\sum_{a < j} \sign(v_{aj})|v_{aj}|^{p - 1}v_{bj} + \sum_{a > j}  \sign(v_{ja})|v_{ja}|^{p - 1}v_{jb}}{\mathcal{G}^{p - 1}(v_{12},v_{13},v_{14},v_{23},v_{24},v_{34})} = \frac{\sum_{j} \sign(v_{aj})|v_{aj}|^{p - 1}v_{bj}}{\mathcal{G}^{p - 1}(v_{12},v_{13},v_{14},v_{23},v_{24},v_{34})}.
\]
Motivated by this expression, we define the even finite measure $\mu'$ on $G$:
\[
\mu'(f) = \int_{\mathbb{G}(4,2)}\frac{f(v_1\wedge v_2)}{\Phi^{p -1}(v_1\wedge v_2)}d\mu(v_1\wedge v_2),\qquad \mbox{for all $f \in C(\mathbb{G}(4,2))$}.
\]
 The goal now is therefore to show that
\[
A(\mu) = \sigma(\mu') := \int B_{\mathcal{G}}(v_1\wedge v_2)\Phi^{p -1}(v_1\wedge v_2) d\mu'
\]
cannot have rank smaller than $2$. For ease of notation, denote
\[
a_{ij}:= \int |v_{ij}|^pd\mu',
\]
for all $i,j \in \{1,\dots,4\}$. We observe that $a_{ij}= a_{ji}$ if $i > j$ and $a_{ii} = 0$. The principal $2\times 2$ subminor $S_{ij}$ of $\sigma(\mu')$, i.e. the submatrix of $\sigma(\mu')$ obtained by considering only the $i$-th and $j$-th rows and columns, are:
\[
S_{ij}= \left(\begin{array}{cc} a_{1i}+ a_{2i} + a_{3i} + a_{4i} & \sum_{s = 1}^4\displaystyle\int \sign(v_{is})|v_{is}|^{p-1}v_{js}d\mu'\\ \sum_{s = 1}^4\displaystyle\int \sign(v_{js})|v_{js}|^{p-1}v_{is}d\mu' & a_{1j}+ a_{2j} + a_{3j} + a_{4j} \end{array}\right).
\]
It suffices to prove that at least one of the $S_{ij}$ must be invertible to conclude the proof. We first use H\"older inequality to estimate
\begin{align*}
|(S_{ij})_{21}|&=\left|\sum_{s = 1}^4\int \sign(v_{js})|v_{js}|^{p-1}v_{is}d\mu'\right|=  \left|\sum_{s\neq i,s\neq j}\int \sign(v_{js})|v_{js}|^{p-1}v_{is}d\mu'\right|\\
& \le \sum_{s\neq i,s\neq j}\left(\int|v_{js}|^pd\mu'\right)^\frac{p-1}{p}\left(\int|v_{is}|^pd\mu'\right)^\frac{1}{p}  = \sum_{s\neq i,s\neq j}a_{js}^{\frac{p-1}{p}}a_{is}^\frac{1}{p}
\end{align*}
and analogously,
\begin{align*}
|(S_{ij})_{12}|=\left|\sum_{s = 1}^4\int \sign(v_{is})|v_{is}|^{p-1}v_{js}d\mu'\right| \le \sum_{s\neq i,s\neq j}a_{is}^{\frac{p-1}{p}}a_{js}^\frac{1}{p}.
\end{align*}
Thus, we can estimate:
\begin{equation}\label{S}
\begin{split}
\det(S_{ij}) &= (S_{ij})_{11}(S_{ij})_{22}-(S_{ij})_{12}(S_{ij})_{21} \ge (S_{ij})_{11}(S_{ij})_{22}-|(S_{ij})_{21}(S_{ij})_{12}| \\
&\ge \sum_{s \neq i}a_{is}\sum_{s\neq j}a_{js} - \sum_{s\neq i,s\neq j}a_{is}^\frac{p-1}{p}a_{js}^\frac{1}{p}\sum_{s\neq i,s\neq j}a_{js}^\frac{p-1}{p}a_{is}^\frac{1}{p}.
\end{split}
\end{equation}
We use again H\"older inequality to write
\begin{equation}\label{in}
\sum_{s\neq i,s\neq j}a_{is}^\frac{p-1}{p}a_{js}^\frac{1}{p} \le \left(\sum_{s\neq i,s\neq j}a_{is}\right)^\frac{p - 1}{p}\left(\sum_{s\neq i,s\neq j}a_{js}\right)^\frac{1}{p} \le \left(\sum_{s\neq i}a_{is}\right)^\frac{p - 1}{p}\left(\sum_{s\neq j}a_{js}\right)^\frac{1}{p}
\end{equation}
and analogously
\begin{equation}\label{in1}
\sum_{s\neq i,s\neq j}a_{is}^\frac{1}{p}a_{js}^\frac{p-1}{p} \le \left(\sum_{s\neq i}a_{is}\right)^\frac{1}{p}\left(\sum_{s\neq j}a_{js}\right)^\frac{p-1}{p}.
\end{equation}
Plugging \eqref{in}-\eqref{in1} in \eqref{S}, we get $\det(S_{ij})\ge 0$. If $\det(S_{ij})=0$, then \eqref{in} would hold with an equality, implying $a_{ij} = 0$. If we had $\det(S_{ij}) = 0$ for every $i,j$, then this would mean
\[
a_{ij} = \int |v_{ij}|^pd\mu' = 0.
\]
This is impossible since $\|\tau\|^2 = \|v_1\wedge v_2\|^2 = \sum_{i<j}|v_{ij}|^2 = 1$ $\mu'$-a.e., that is a consequence of the fact that $\mu$ is supported on $G$. Hence at least one of the $S_{ij}$ must be invertible, as desired.
\end{proof}

\begin{remark}
Let us remark that the considerations at the beginning of the section and Lemmas \ref{Beq}-\ref{char} hold in $\mathbb{G}(N,m)$, but we have chosen here to use only $\mathbb{G}(4,2)$ for ease of notation. We do not know whether Theorem \ref{Lp} holds for higher dimensions or codimensions. Moreover, numerical simulations indicated that for $\mathcal{G}(\cdot) = \|\cdot\|_{\ell_p}, p \in (1,\infty)$, $\Psi_{\mathcal{G}}$ fulfills (SAC), hence actually satisfies (AC), by Proposition \ref{PROOFAC2}. Unfortunately, we were not able to prove it analytically.
\end{remark}

\section{Acknowledgements}
The first author has been partially supported by the NSF DMS Grant No.~1906451. The second author has been supported by the SNF Grant 182565.

\bibliographystyle{plain}
\bibliography{ARpaper}

\end{document}